\documentclass[11pt]{article}

\makeatletter

\usepackage{amssymb,amsmath,amsthm,latexsym}
\usepackage[mathscr]{eucal} \usepackage{mathrsfs}
\usepackage{color}
\usepackage{cases}
\usepackage{hyperref}
\usepackage{comment}
\usepackage{mathtools}
\let\wfs@comment@comment\comment
\let\comment\@undefined

\usepackage{changes}
\let\wfs@changes@comment\comment
\let\comment\@undefined

\newcommand\comment{%
    \ifthenelse{\equal{\@currenvir}{comment}}
    {\wfs@comment@comment}
    {\wfs@changes@comment}%
}
\definechangesauthor[name=Daniele, color=red]{DAN}
\definechangesauthor[name=Giuseppe, color=green]{GIUS}
\title{Saturating linear sets of minimal rank}
\usepackage{authblk}
\author[1]{Daniele Bartoli}
\affil[1]{Department of Mathematics and Informatics, University of Perugia, Perugia,  Italy, {\small \texttt{daniele.bartoli@unipg.it}}\vspace*{.3cm}}
\author[2]{Martino Borello}
\affil[2]{Universit\'e Paris 8, Laboratoire de G\'eom\'etrie, Analyse et Applications, LAGA,
Universit\'e Sorbonne Paris Nord, CNRS, UMR 7539, France, {\small\texttt{martino.borello@univ-paris8.fr}}\vspace*{.3cm}}
\author[3]{Giuseppe Marino}
\affil[3]{Department of Mathematics and Applications ``R. Caccioppoli'', University of Naples Federico II, Napoli, Italy, {\small\texttt{giuseppe.marino@unina.it}}\vspace*{.3cm}}
\date{}

\DeclareMathOperator{\rk}{rk}

\begin{document}
\maketitle

\newtheorem{theorem}{Theorem}[section]
\newtheorem{lemma}[theorem]{Lemma}
\newtheorem{conj}[theorem]{Conjecture}
\newtheorem{remark}[theorem]{Remark}
\newtheorem{corollary}[theorem]{Corollary}
\newtheorem{prop}[theorem]{Proposition}
\newtheorem{definition}[theorem]{Definition}
\newtheorem{result}[theorem]{Result}
\newtheorem{property}[theorem]{Property}

\newtheorem{question}[theorem]{Question}

\makeatother
\newcommand{\Prf}{\noindent{\bf Proof}.\quad }
\renewcommand{\labelenumi}{(\alph{enumi})}

\def\B{\mathbf B}
\def\C{\mathbf C}
\def\Q{\mathbf Q}
\def\W{\mathbf W}
\def\a{\mathbf a}
\def\b{\mathbf b}
\def\c{\mathbf c}
\def\d{\mathbf d}
\def\e{\mathbf e}
\def\l{\mathbf l}
\def\v{\mathbf v}
\def\w{\mathbf w}
\def\x{\mathbf x}
\def\y{\mathbf y}
\def\z{\mathbf z}
\def\t{\mathbf t}
\def\cD{\mathcal D}
\def\cC{\mathcal C}
\def\cV{\mathcal V}
\def\cH{\mathcal H}
\def\cM{{\mathcal M}}
\def\cK{\mathcal K}
\def\cU{U}
\def\cS{\mathcal S}
\def\cT{\mathcal T}
\def\cR{\mathcal R}
\def\cN{\mathcal N}
\def\cA{\mathcal A}
\def\cF{\mathcal F}
\def\cL{\mathcal L}
\def\cB{\mathcal B}
\def\cP{\mathcal P}
\def\cG{\mathcal G}
\def\cGD{\mathcal GD}
\def\mC{\mathcal C}
\def\mU{U}
\def\mH{\mathcal H}
\def\PG{{\rm PG}}
\def\GF{{\rm GF}}
\def\Tr{{\rm Tr}}
\def\Pg{PG(5,q)}
\def\pg{PG(3,q^2)}
\def\ppg{PG(3,q)}
\def\HH{{\cal H}(2,q^2)}
\def\F{\mathbb F}
\def\Ft{\mathbb F_{q^t}}
\def\P{\mathbb P}
\def\Z{\mathbb Z}
\def\V{\mathbb V}
\def\bS{\mathbb S}
\def\G{\mathbb G}
\def\E{\mathbb E}
\def\N{\mathbb N}
\def\K{\mathbb K}
\def\D{\mathbb D}
\def\ps@headings{
 \def\@oddhead{\footnotesize\rm\hfill\runningheadodd\hfill\thepage}
 \def\@evenhead{\footnotesize\rm\thepage\hfill\runningheadeven\hfill}
 \def\@oddfoot{}
 \def\@evenfoot{\@oddfoot}
}
\def\cub{\mathscr C}
\def\cO{\mathcal O}
\def\cur{\mathscr L}
\def\Fqm{{\mathbb F}_{q^m}}
\def\Fq3{{\mathbb F}_{q^3}}
\def\fq{{\mathbb F}_{q}}
\def\Fm{{\mathbb F}_{q^m}}
\def\Fn{{\mathbb F}_{q^n}}

\newcommand{\Fmk}{[n,k]_{q^m/q}}
\newcommand{\Fmkd}{[n,k,d]_{q^m/q}}
\newcommand{\Fmkdd}{[n,k,(d_1,\ldots,d_k)]_{q^m/q}}
\newcommand{\ale}[1]{{\color{blue}[$\star \star$ {\sf Alessandro: #1}]}}
\newcommand{\giu}[1]{{\color{red}[$\star \star$ {\sf Giuseppe: #1}]}}
\newcommand{\Fms}[3]{[#1,#2]_{q^{#3}/q}}

\begin{abstract}
Saturating sets are combinatorial objects in projective spaces over finite fields that have been intensively investigated in the last three decades. They are related to the so-called covering problem of codes in the Hamming metric. In this paper, we consider the recently introduced linear version of such sets, which is, in turn, related to the covering problem in the rank metric. The main questions in this context are how small the rank of a saturating linear set can be and how to construct saturating linear sets of small rank. Recently, Bonini, Borello, and  Byrne provided a lower bound on the rank of saturating linear sets in a given projective space, which is shown to be tight in some cases. In this paper, we provide construction of saturating linear sets meeting the lower bound and we develop a link between the saturating property and the scatteredness of linear sets. The last part of the paper is devoted to show some parameters for which the bound is not tight.\\


\end{abstract}

\bigskip

\par\noindent

\section*{Introduction}

A set $S\subseteq \PG(k-1,q^m)$ is called \emph{$\rho$-saturating} if every point of $\PG(k-1,q^m)$ lies in a subspace generated by $\rho+1$ points of $S$. The term saturated was coined by \cite{ughi1987saturated}, but used with a slightly different meaning. The above definition comes from \cite{pambianco1996small} and it has been consolidated in \cite{davydov2000saturating}. There is a classical correspondence between $\rho$-saturating sets and linear codes with covering radius $\rho+1$, obtained by considering the vector representatives of $S$ as the columns of a parity check matrix of a linear code. It is natural to wonder how small a  $\rho$-saturating set in $\PG(k-1,q^m)$ can be. This is equivalent, in the coding-theoretical language, to the so-called \emph{covering problem} (see \cite[Chapter 1]{cohen1997covering}). Many bounds and construction of small saturating sets has been given over the last two decades (see for example \cite{davydov2003saturating,denaux2022constructing,nagy2018saturating} and references therein). 

If $m\geq 2$, instead of simply considering sets of points, we may look at linear sets: a linear set associated to an $\F_q$-vector space $U$ in $\F_{q^m}^k$ is the set of points 
\[L_\mU:=\{\langle u \rangle_{\F_{q^m}} :  u\in \mU\setminus\{0\}\}\subseteq \PG(k-1, q^m),\]
where $\langle u \rangle_{\F_{q^m}}$ denotes the projective point corresponding to $u$.
Such objects were introduced in \cite{lunardon1999normal} in order to construct blocking sets and they have been the subject of intense research over the last years. The $\F_q$-dimension of $U$ is also called \emph{rank} of the linear set. Intuitively, saturating linear sets would not be the smallest ones, in fact, they would be quite large in cardinality. However, a natural question is how small the rank of a $\rho$-saturating linear set in $\PG(k-1, q^m)$ can be. As shown in \cite{bonini2022saturating}, this is equivalent to the covering problem for rank-metric codes. Let us underline that the knowledge of the covering properties of a rank-metric code has important consequences for applications: the \emph{covering radius} is the least integer $\rho$ such that every element of the ambient space is in a ball of radius $\rho$ centered in some codeword. It measures the maximum weight of any correctable error in the ambient space and it characterizes the maximality of the code (namely, if the code is contained in another with the same minimum distance). See \cite{byrne2017covering} for more details. In general, it is much harder to compute the covering radius than the minimum distance of a code. While there is a wide literature on the covering problem in the Hamming metric, there are relatively few papers on the subject in the rank-metric case \cite{byrne2017covering,byrne2019assmus,gadouleau2008packing,gadouleau2009bounds}. 

In the current work, we continue the investigation of the geometrical approach to such problem introduced in \cite{bonini2022saturating}. Let $s_{q^m/q}(k, \rho)$ denotes the smallest rank of a $(\rho-1)$-saturating linear set in $\PG(k-1, q^m)$. In \cite{bonini2022saturating}, it is proved that 
\begin{equation}\label{eq:lb}
s_{q^m/q}(k, \rho) \geq \begin{dcases}
           \left\lceil \frac{mk}{\rho}\right\rceil - m + \rho & \text{ if } q>2,\\
           \left\lceil \frac{mk-1}{\rho}\right\rceil - m + \rho & \text{ if } q=2, \rho> 1,\\
           m(k-1)+1 & \text{ if } q=2,\rho =1.
\end{dcases}
\end{equation}
and it is shown that \eqref{eq:lb} is tight for some values of $q,m,k$ and $\rho$. In this paper, we first extend the results in \cite{bonini2022saturating} by proving that \eqref{eq:lb} is tight whenever $\rho$ divides $k$. Secondly, we show some covering properties of $h$-scattered linear sets of large rank, which are particular linear sets introduced first in \cite{csajbok2021generalising}. As a byproduct, thanks to known results on the existence of maximum $h$-scattered linear sets, we get other saturating linear sets of small rank and, in some cases not covered by previous results, we get the tightness of the bound \eqref{eq:lb}. The last part of the paper is devoted to the proof that there exist values of $q,m,k$ and $\rho$ for which the bound \eqref{eq:lb} is not met, clarifying then that the bound \eqref{eq:lb} is not always tight. 

\bigskip

\textbf{Outline:} Section \ref{sec:preliminaries} is devoted to define and give preliminary results about the main objects of the paper. In Section \ref{sec:smallsaturating} we provide constructions of small saturating linear sets showing the tightness of the lower bound \eqref{eq:lb} for some parameters, whereas in Section \ref{sec:non-tightness} we prove that the bound is not always tight. Finally, in Section \ref{sec:conclusion} we resume our results and present some open problems.

\section{Preliminaries}\label{sec:preliminaries}

Throughout this paper, $q$ is a prime power, $[t]$ denotes the set $\{1,2,\ldots,t\}\subseteq \Z$, $k\in [n]$, $m\geq 2$, $\rho\in[\min\{k,m\}]$, $\F_q$ denotes the finite field of order $q$ and $V(k,q^m)$ is a vector space of dimension $k$ over $\F_{q^m}$. Also, ${\rm N}_{q^m/q}(z)$ and ${\rm Tr}_{q^m/q}(z)$ denote the (standard) norm and trace  from $\mathbb{F}_{q^m}$ to $\mathbb{F}_q$ of $z\in \mathbb{F}_{q^m}$.

The projective geometry $\PG(k-1,q)$ with underlying vector space $\F_q^k$ is 
$$ \PG(k-1,q):= \left(\F_q^{k}\setminus \{0\}\right)/_\sim, $$
where, for $u,v\in \F_q^{k}\setminus \{0\}$, $u\sim v$ if and only if $u=\lambda v$
for some $\lambda \in \F_q\setminus \{0\}$.

\begin{definition}
Let $\cS\subseteq \mathrm{PG}(k-1,q^m)$.
\begin{enumerate}
    \item[{\rm (a)}] A point $Q \in \mathrm{PG}(k - 1,q^m)$ is said to be $(\rho-1)$-saturated by $\cS$ if there exist $\rho$ points $P_1,\ldots,P_{\rho}\in \cS$ such that $Q \in \langle P_1,\ldots,P_{\rho}\rangle_{\mathbb{F}_{q^m}}$. We also say that $\cS$ $\rho$-covers $Q$.
    \item [{\rm (b)}] The set $\cS$ is $(\rho-1)$-saturating set of $\mathrm{PG}(k - 1, q^m)$ if every point $Q \in \mathrm{PG}(k - 1, q^m)$ is $(\rho-1)$-saturated by $\cS$ and $\rho$ is the smallest value with this property.
\end{enumerate}
\end{definition}

In order to introduce the $q$-analogue of the above definition, we need to introduce the notion of linear sets. These are combinatorial  objects, introduced by Lunardon in \cite{lunardon1999normal}, which are subject of intense research over the last two decades. A thorough presentation of linear sets can be found in  \cite{polverino2010linear}. 

\begin{definition}
Let $\mU$ be an $\F_q$-subspace of $V(k,q^m)$ of $\F_q$-dimension $n$. The $\F_q$-linear set in $\PG(k-1, q^m)$ of rank $n$ associated to $\mU$ is the set $$L_\mU:=\{\langle u \rangle_{\F_{q^m}} :  u\in \mU\setminus\{0\}\},$$
where $\langle u \rangle_{\F_{q^m}}$ denotes the projective point corresponding to $u$. If the size of $L_\mU$ is maximal, i.e.  $|L_\mU|=\frac{q^n-1}{q-1}$, then $L_\mU$ is called scattered.
\end{definition}

\begin{definition}
    Let $\Lambda = \PG(T,q^m)$ be a projective subspace of $\PG(k-1,q^m)$ with underlying vector space $T \subseteq V(k,q^m)$. Then $L_{\mU \cap T}=L_\mU \cap \Lambda,$ and the \textit{weight} of $\Lambda$ in $L_\mU$ is defined as
\[w_{L_\mU}
(\Lambda) = \dim_{\mathbb{F}_q}(\cU \cap T). \]
If $\text{dim}_{\mathbb{F}_q}(\mU \cap T)=i$, one shall say that $\Lambda$ has {\it weight} $i$ in $L_\mU$.
\end{definition}

We are now ready to define the main object of the paper, introduced first in \cite{bonini2022saturating}.

\begin{definition}
An $\F_q$-subspace $\mU$ in $V(k,q^m)$ is a rank $\rho$-saturating system if $L_\mU$ is a $(\rho - 1)$-saturating set in $\mathrm{PG}(k - 1, q^m)$. This last is called a $(\rho - 1)$-saturating linear set.
\end{definition}

For an $\F_q$-subspace $\mU$ in $V(k,q^m)$ of $\F_q$-dimension $n$, let us consider a $k\times n$ matrix $G$ whose columns are the elements of an $\F_q$-basis of $\mU$. The $\F_{q^m}$-vector space $\mC$ generated by its rows is called a code associated to $\mU$. The dual code $\mC^\perp$ is the orthogonal space with respect to the standard inner product.

\begin{definition}
The rank covering radius of a code $\mC \leq \F_{q^m}^n$ is the integer
$$\rho_{\rk}(\mC):= \max \{ \min \{ \dim_{\F_q}\langle x_1-c_1, x_2-c_2,\ldots, x_n-c_n\rangle_{\F_q} : c \in \mC\} : x \in \F_{q^m}^n\}.$$
\end{definition}

The following relation holds between the rank covering radius and rank saturating systems.

\begin{theorem}[{\cite[Theorem 2.5]{bonini2022saturating}}]
Let $\cU$ be an $\F_q$-subspace in $V(k,q^m)$ and $\cC$ a code associated to $\cU$. Then, $\mU$ is a rank $\rho$-saturating system if and only if $\rho_{\rk}(\cC^\perp)=\rho$.
\end{theorem}

From both a geometric and a coding theoretical point of view, it is meaningful to ask how small the $\F_q$-dimension of a rank $\rho$-saturating system in $V(k,q^m)$ can be. Hence, let us introduce the following notation. 

\begin{definition}[\!\! \cite{bonini2022saturating}]
For fixed $\rho, q,k,m$,  $s_{q^m/q}(k, \rho)$ denotes the minimal $\mathbb{F}_q$-dimension of a rank $\rho$-saturating system in $V(k,q^m)$, or, equivalently, smallest rank of a $(\rho-1)$-saturating linear set in $\PG(k-1, q^m)$.
\end{definition}

In \cite{bonini2022saturating}, the following bounds are proved.

\begin{theorem}[{\cite[Theorems 3.3 and 3.4]{bonini2022saturating}}]
The function $s_{q^m/q}(k, \rho)$ satisfies the bounds
\begin{align}
    s_{q^m/q}(k, \rho) & \geq \begin{dcases}
           \left\lceil \frac{mk}{\rho}\right\rceil - m + \rho & \text{ if } q>2,\\
           \left\lceil \frac{mk-1}{\rho}\right\rceil - m + \rho & \text{ if } q=2, \rho> 1,\\
           m(k-1)+1 & \text{ if } q=2,\rho =1.
\end{dcases}\\ \medskip
s_{q^m/q}(k, \rho) & \leq m(k-\rho)+\rho.
\end{align}

\end{theorem}

\begin{remark}
Let us resume here all other known properties of $s_{q^m/q}(k,\rho)$. By \cite[Theorem 3.6]{bonini2022saturating} and \cite[Theorem 3.8]{bonini2022saturating}, the following holds, for all positive integers $m,k,k'$, $\rho\in [\min\{k,m\}]$, $\rho'\in [\min\{k',m\}]$.
\begin{itemize}
    \item[{\rm (a)}] If $\rho<\min\{k,m\}$, then $s_{q^m/q}(k,\rho+1) \leq s_{q^m/q}(k,\rho).$
    \item[{\rm (b)}] $s_{q^m/q}(k,\rho)<s_{q^m/q}(k+1,\rho)$.
    \item[{\rm (c)}] If $\rho<m$, then $s_{q^m/q}(k+1,\rho+1) \leq s_{q^m/q}(k,\rho)+1$.
    \item[{\rm (d)}] If $\rho+\rho'\leq \min\{k+k',m\}$, $s_{q^m/q}(k+k',\rho+\rho') \leq s_{q^m/q}(k,\rho)+s_{q^m/q}(k',\rho')$.
\end{itemize}
Using linear cutting blocking sets, introduced in \cite{alfarano2022linear}, one gets
\[s_{q^{r(k-1)}/q}(k,k-1)\leq 2k+r-2,\]
(see \cite[Corollary 4.7.]{bonini2022saturating}) and
$s_{q^{2r}/q}(3,2)\leq r+3$ for all $r\geq 4$ (an easy consequence of  \cite[Theorem 7.16]{gruica2022generalised}).

Using subgeometries (see \cite[Theorem 4.14.]{bonini2022saturating}), one gets
\[s_{q^{tr}/q}(t(r-1)+1+h,t(r-1)+1)\leq th+t(r-1)+1,\]
for $t,s\geq 2$, $h\geq 0$.\\
Finally, $s_{q^m/q}(k,\rho)$ is determined in the following case (see \cite[Section 5]{bonini2022saturating}):
\begin{align*}
    s_{q^m/q}(k,1) & = m(k-1)+1, & \text{ for all }m,k\geq 2,  \\
    s_{q^m/q}(k,k) & = k, & \text{ for all }m,k\geq 2,\\
    s_{q^{2r}/q}(3,2) & = r+2, & \text{ for } r\neq 3,5\bmod 6\text{ and }r\geq 4,
    \\
    s_{q^{2r}/q}(3,2) & = r+2, & \text{ for } \gcd(r,(q^{2s}-q^s+1)!)=1, r\text{ odd}, 1\leq s\leq r, \gcd(r,s)=1 \ \text{(see \cite{lia2023short})},
    \\
    s_{q^{10}/q}(3,2) & = 7, & \text{ for } q=p^{15h+s}, p\in\{2,3\}, \gcd(s,15)=1 \ (see \cite{bartoli2021evasive}),\\
     s_{q^{10}/q}(3,2) & = 7, & \text{ for } q=5^{15h+1}, \ (see \cite{bartoli2021evasive}),\\
    s_{q^{10}/q}(3,2) & = 7, & \text{ for } q \text{\ odd}, q= 2,3\bmod 5 \text{ and for } q=2^{2h+1}, h\geq 1,\ (see \cite{lia2023short}),
    \\
    s_{q^{2r}/q}(2r,2r-1)  &=2r+1, & \text{ for all }r\geq 2.
\end{align*}
\end{remark}

\section{Saturating systems meeting the lower bound}\label{sec:smallsaturating}

In this section we aim to present saturating systems of minimal $\F_q$-dimension.

The first construction is based essentially on the notion of Moore matrix. Let us recall that a square matrix over $\F_q$ is a Moore matrix if it has successive powers of the Frobenius automorphism applied to its columns. It is invertible if and only if the elements in the left hand column are linearly independent over $\F_q$.

\begin{theorem}\label{Th:main}
Let $k=\rho t$ for some integer $t\geq 1$. Then 
\[\cU:=\{(x_1,x_1^q,\ldots,x_1^{q^{\rho-1}},x_2,x_2^q,\ldots,x_2^{q^{\rho-1}},\ldots,x_{t-1},x_{t-1}^q,\ldots,x_{t-1}^{q^{\rho-1}},a_1,a_2,\ldots,a_\rho)^T: x_i\in\F_{q^m}, a_j\in\F_q\}.\]
is rank $\rho$-saturating. Therefore,
\[s_{q^m/q}(\rho t,\rho)= m(t-1)+\rho.\]
\end{theorem}
\begin{proof} 
Let $r=t-1$ and consider the $\mathbb{F}_q$-vector space
\[\cU:=\{(x_1,x_1^q,\ldots,x_1^{q^{\rho-1}},x_2,x_2^q,\ldots,x_2^{q^{\rho-1}},\ldots,x_r,x_r^q,\ldots,x_r^{q^{\rho-1}},a_1,a_2,\ldots,a_\rho)^T: x_i\in\F_{q^m}, a_j\in\F_q\}.\]
Let \[v=(Y_1^{(1)},\ldots,Y_\rho^{(1)},Y_1^{(2)},\ldots,Y_\rho^{(2)},\ldots, Y_1^{(r)},\ldots,Y_\rho^{(r)},A_1,\ldots,A_\rho)^T\in V(k,q^m).\] We want to determine 
\[u_i=(x_1^{(i)},\ldots,(x_1^{(i)})^{q^{\rho-1}},x_2^{(i)},\ldots,(x_2^{(i)})^{q^{\rho-1}},\ldots,x_r^{(i)},\ldots,(x_r^{(i)})^{q^{\rho-1}},a_1^{(i)},a_2^{(i)},\ldots,a_\rho^{(i)})^T\in \cU\]
such that 
$$v\in \langle u_1,\ldots,u_\rho\rangle_{\F_q^m}.$$
Let $s\in \{0,\ldots,\rho\}$ be the $\mathbb{F}_q$-rank of $A_1,\ldots,A_\rho$.\\
Without loss of generality we can suppose that $A_1,\ldots,A_s$ are $\mathbb{F}_q$-linear independent and 
$$A_j=\sum_{h=1}^{s}\alpha_j^{(h)}A_h,$$
for $j\in\{s+1,\ldots,\rho\}$ and $\alpha_j^{(h)}\in \mathbb{F}_q$. Set $\lambda_i=A_i$, for $i\in[s]$ and choose

\begin{equation}\label{Eq:a_i}
\begin{pmatrix}
a_1^{(1)}&a_2^{(1)}&\cdots & a_\rho^{(1)}\\
a_1^{(2)}&a_2^{(2)}&\cdots & a_\rho^{(2)}\\
\vdots&\vdots&&\vdots\\
a_1^{(k)}&a_2^{(k)}&\cdots & a_\rho^{(k)}
\end{pmatrix}
=
\begin{pmatrix}
1&0&\cdots&0 & \alpha_{s+1}^{(1)}&\alpha_{s+2}^{(1)} &\cdots &\alpha_{k}^{(1)}\\
0&1&\cdots&0&  \alpha_{s+1}^{(2)}&\alpha_{s+2}^{(2)} &\cdots &\alpha_{\rho}^{(2)}\\
\vdots&\vdots&&\vdots&\vdots&\vdots&&\vdots\\
0&0&\cdots&1 & \alpha_{s+1}^{(s)}&\alpha_{s+2}^{(s)} &\cdots &\alpha_{\rho}^{(s)}\\\\
0&0&\cdots&0 & 0&0&\cdots&0\\
\vdots&\vdots&&\vdots&\vdots&\vdots&&\vdots\\
0&0&\cdots&0 & 0&0&\cdots&0\\
\end{pmatrix}.
\end{equation}
Consider $\lambda_{s+1},\ldots,\lambda_{\rho}$ such that $A_1,\ldots,A_s,\lambda_{s+1},\ldots,\lambda_{\rho}$ are $\mathbb{F}_q$-linear independent.\\
Thus for each $j\in[r]$ there exists a unique  $(x_j^{(1)},\ldots,x_j^{(\rho)})^T\in \mathbb{F}_{q^m}^\rho$ such that  
$$
(x_j^{(1)},\ldots,x_j^{(\rho)})
\begin{pmatrix}
\lambda_1&\sqrt[q]{\lambda_1}&\ldots&\sqrt[q^{\rho-1}]{\lambda_1}\\
\lambda_2&\sqrt[q]{\lambda_2}&\ldots&\sqrt[q^{\rho-1}]{\lambda_2}\\
\vdots&\vdots&&\vdots\\
\lambda_\rho&\sqrt[q]{\lambda_\rho}&\cdots&\sqrt[q^{\rho-1}]{\lambda_\rho}\\
\end{pmatrix}=\left(Y_1^{(j)},\sqrt[q]{Y_2^{(j)}},\ldots,\sqrt[q^{\rho-1}]{Y_\rho^{(j)}}\right),$$
i.e.
$$
(\lambda_1,\ldots,\lambda_\rho)\begin{pmatrix}
x_j^{(1)}&(x_j^{(1)})^q&\ldots&(x_j^{(1)})^{q^{\rho-1}}\\
x_j^{(2)}&(x_j^{(2)})^q&\ldots&(x_j^{(2)})^{q^{\rho-1}}\\
\vdots&\vdots&&\vdots\\
x_j^{(\rho)}&(x_j^{(\rho)})^q&\ldots&(x_j^{(\rho)})^{q^{\rho-1}}\\
\end{pmatrix}=\left(Y_1^{(j)},Y_2^{(j)},\ldots,Y_\rho^{(j)}\right).$$

This, together with \eqref{Eq:a_i}, provides a choice for $u_i$, $i\in[\rho]$ such that $v\in \langle u_1,\ldots,u_\rho\rangle_{\F_{q^m}}$ and shows that any vector $v\in V(k,q^m)$ is $\rho$-saturated.

Such a $\rho$ is minimal since \[s_{q^m/q}(\rho t,\rho-1)\geq \frac{\rho}{\rho-1}\cdot m( t-1)+\frac{m}{\rho-1}+\rho-1 >m(t-1)+\rho,\] 
for $q>2$, and similarly for $q=2$.
\end{proof}

\begin{remark}\label{Remark}
Observe that in Theorem \ref{Th:main} one can also consider the set
\[\cU:=\{(x_1,x_1^{q^{s_1}},\ldots,x_1^{q^{s_1(\rho-1)}},\ldots,x_r,x_r^{q^{s_r}},\ldots,x_r^{q^{s_r(\rho-1)}},a_1,a_2,\ldots,a_\rho)^T: x_i\in\F_{q^m}, a_j\in\F_q\},\]
with $\gcd(m,s_j)=1$ and prove with the same arguments that $\cU$ is rank $\rho$-saturating.
\end{remark}

\bigskip

The second construction is based on  the following generalization of scattered linear sets, introduced first in \cite{csajbok2021generalising}.

\begin{definition}
Let $U$ be an $\F_q$-subspace of $V(k,q^m)$
and $h<k$ be a positive integer. Then $U$ is called $h$-scattered if $\langle \cU\rangle_{\F_{q^m}}=V$ and for every  $\F_{q^m}$-subspace $W$ of $V$ of dimension $h$, $\dim_{\F_q}(W\cap \cU)\leq h$.
\end{definition}

The $\F_q$-dimension of an $h$-scattered subspace is upper bounded as follows.

\begin{theorem}[{\cite[Theorem 2.3.]{csajbok2021generalising}}]\label{Th:csajbok2021generalising}
Let $\cU$ be an $h$-scattered $\F_q$-subspace of $V(k,q^m)$. Then either
\begin{equation}\label{eq:upperboundhscattered}\dim_{\F_q}\cU\leq \left\lfloor \frac{km}{h+1}\right\rfloor,
\end{equation}
or $\dim_{\F_q} \cU=k$ and $\cU$ defines a subgeometry of ${\rm PG}(k-1,q^m)$ and it is $(k-1)$-scattered.
\end{theorem}

\begin{definition}
    An $h$-scattered $\F_q$-subspace of $V(k,q^m)$ whose $\F_q$-dimension meets the bound~\eqref{eq:upperboundhscattered} is called maximum $h$-scattered.
\end{definition}

\begin{remark}\label{rmk:existencehscattered}
The upper bound \eqref{eq:upperboundhscattered} is known to be achieved in the following cases:

\begin{itemize}
    \item[{\rm (a)}] $h=1$ and $mk$ is even, see \cite{ball2000linear,bartoli2018maximum,blokhuis2000scattered,csajbok2017maximum};
    \item[{\rm (b)}] $h=1$, $k=3$, $m=3$, see \cite{bartoli2021evasive};
    \item[{\rm (c)}] $h=1$, $k=3$, $m=5$, $q=p^{15t+s}$ with $p\in \{2,3\}$ and $\gcd(s,15)=1$ or $q=5^{15t+1}$, see \cite{bartoli2021evasive};
    \item[{\rm (d)}] $h=1$, $k=3$, $m=5$, $q$ odd and $q=2,3\mod 5$ or $q=2^{2t+1}$ with $t\geq 1$, see \cite{lia2023short};
    \item[{\rm (e)}] $h=m-3$ and $m\geq 4$ is even and $k=r(m-2)/2$ where $r\geq 3$ is odd, see \cite[Theorem 3.6]{csajbok2021generalising};
    \item[{\rm (f)}] $h=k-1$ and $k\leq m$, see \cite[Lemma 2.2]{csajbok2021generalising};
    \item[{\rm (g)}] $(h+1)|k$ and $m\geq h+1$, see \cite{csajbok2021generalising,napolitano2021linear}.
\end{itemize}
\end{remark}

\begin{theorem}\label{thm:saturatingscattered}
Let $m\geq h+1$. If $\cU$ is an $h$-scattered $\F_q$-subspace of $V(k,q^m)$ of $\F_q$-dimension (at least) $\left\lfloor\frac{m(k-1)}{h+1}\right\rfloor+1$, then $\cU$ is rank $\rho$-saturating, with $\rho\leq h+1$.
\end{theorem}

\begin{proof}
Let $P:=\langle v\rangle_{\F_{q^m}}$ with $P\notin L_\cU$ and project $\cU$ from $P$ to a hyperplane $\cH$ of $V(k,q^m)$ not containing $\langle v \rangle_{\mathbb{F}_{q^m}}$. Let $\overline{\cU}$ be such a projection. Then $\overline{\cU}$ is a subspace of $\cH$ of $\F_q$-dimension $\left\lfloor\frac{m(k-1)}{h+1}\right\rfloor+1$ which is not $h$-scattered since its dimension exceeds the upper bound in Theorem \ref{Th:csajbok2021generalising}. Hence there exists an $\mathbb{F}_{q^m}$-subspace $\cM$ of $\cH$ of $\F_{q^m}$-dimension $h$ such that $\dim_{\F_q}(\cM\cap\overline{\cU})\geq h+1$. Let $\mathcal{N}=\langle v,\cM\rangle_{\F_{q^m}}$, which is clearly of $\F_{q^m}$-dimension $h+1$ and such that $\dim_{\F_q}(\mathcal{N}\cap \cU)\geq h+1$. Let $u_1,\ldots,u_{h+1}\in \mathcal{N}\cap \cU$ be $h+1$ linearly independent vectors over $\F_q$. The $\F_{q^m}$-vector space $\langle u_1,\ldots,u_{h+1}\rangle_{\F_{q^m}}$ must have dimension $h+1$, since otherwise we would get a contradiction with $\cU$ being $h$-scattered. So $\mathcal{N}=\langle u_1,\ldots,u_{h+1}\rangle_{\F_{q^m}}$. Hence $v\in \mathcal{N}$ is $(h+1)$-saturated by $\cU$.
\end{proof}

\begin{corollary}\label{Th:hscatt}
Let $m\geq 4$ be an even integer. For $q>2$, if $r=3$ and $m<12$ or $r>3$ odd, then 
$$\frac{mr}2-2\leq s_{q^m/q}\left(\frac{r(m-2)}{2},m-2\right)\leq \frac{mr}2-1.$$
For $q=2$, the same holds if $r=3$ and $m<10$ or $r>3$ odd.
\end{corollary}

\begin{proof}
If $r\geq 3$ is odd and $m\geq 4$ is even, as recalled in Remark \ref{rmk:existencehscattered}, there exists a maximum $(m-3)$-scattered subspace, say $W$, in $V(r(m-2)/2,q^m)$. Let $\cU$ be a subspace of $W$ of dimension \[\left\lfloor\frac{m(\frac{r(m-2)}{
2}-1)}{m-2}\right\rfloor+1=\frac{mr}{2}-1.\] By Theorem \ref{thm:saturatingscattered}, $\cU$ is $\rho$-saturating with $\rho\leq m-2$. We have that $\cU$ cannot be rank $h$-saturating with $h<m-2$, because otherwise its $\F_q$-dimension would be smaller than the lower bound \eqref{eq:lb}. Actually, for $q>2$,
\[\left\lceil\frac{m}{m-2-t}\cdot \frac{r(m-2)}{2}\right\rceil-2-t>\frac{mr}{2}-1\]
for all $0<t<m-2$, when either $r\geq 4$ or $r=3$ and $m<12$.\\
For $q=2$, exactly the same arguments, with the slightly different lower bound, lead to the stated result.
\end{proof}

\begin{corollary}
If $mk$ is even, then 
\[\left\lceil \frac{m(k-2)}{2}\right\rceil+2 \leq  s_{q^m/q}(k,2)\leq \left\lfloor \frac{m(k-1)}{2}\right\rfloor +1=\left\lceil \frac{m(k-2)}{2}\right\rceil+2+\left\lfloor \frac{m}{2} \right\rfloor -1.\]
In particular $s_{q^2/q}(k,2)=k$. Moreover,
\[s_{q^5/q}(3,2)\in \{5,6\}\] 
for $q=p^t$ with $p\in\{2,3,5\}$ and 
\[s_{q^3/q}(3,2)=4.\]
\end{corollary}
\begin{proof}
The proof works exactly as for Corollary \ref{Th:hscatt}, in the case stated in Remark \ref{rmk:existencehscattered} for $h=1$. Proving that $\rho=2$ it is much easier here, with simple inequalities.
\end{proof}

\begin{remark}
Note that Theorem \ref{thm:saturatingscattered} does not give stronger results in the case $\rho|k$.
\end{remark}

\section{Non-tightness of the lower bound}\label{sec:non-tightness}

In this section we will show that the lower bound \eqref{eq:lb} is not tight in general. In order to do it, we will consider $s_{q^4/q}(3,2)$. The lower bound states that 
\[s_{q^4/q}(3,2)\geq 4\]
and it is easy to realize that 
\[s_{q^4/q}(3,2)\leq 5,\]
because an example of rank $2$-saturating system of rank $5$ is obtained considering an $\F_q$-linear sets of rank 5 with a line of weight $4$ (with a scattered underlying space) and a point outside it. {\sc Magma} computational results show that $s_{16/2}(3,2)=s_{81/3}(3,2)=5$, so that the lower bound is not tight in the binary and ternary case. In the rest of the section we aim to generalize this result for infinitely many values of $q$. We will prove the following.

\begin{theorem}\label{TH}
If $q$ is even and large enough, then $s_{q^4/q}(3,2)=5$.
\end{theorem}

Since most of the calculations and results will be based on the parity of $q$, let us \emph{suppose, from now on, that $q$ is even}.

We want to show that for any $\F_q$-linear set $L_U$ of rank $4$ in $\mathrm{PG}(2,q^4)$, there exists a point $P\notin L_U$ for which do not pass any secant line to $L_U$. The proof of the theorem will be divided into four lemmas. The proof of the last three is very technical and it is left to the Appendix.

\medskip

Let $X_0,X_1,X_2$ be homogeneous projective coordinates in $\PG(2,q^4)$ and let $(a_0:a_1:a_2)$ denote the coordinates of a point of the plane.  We will use frequently the following general result.

\begin{remark}\label{rem:projection}
Let $L_U$ be a linear set in $\mathrm{PG}(k-1,q^m)$, $H$ a hyperplane and $P$ a point not belonging to $L_U$ nor to $H$. If the projection of $L_U$ from $P$ to $H$ is scattered, then the point is not $1$-saturated, because otherwise in the projection we would find a point of weight at least $2$.
\end{remark}

\begin{lemma}
Let $L_U$ be an $\F_q$-linear set of rank $4$ in $\mathrm{PG}(2,q^4)$. If $U$ is rank $2$-saturating, then, up to $\mathrm{GL}(3,q^4)$-equivalence,
\[U=U_{\alpha,\beta}=\left\{(x,x^q+\alpha x^{q^3},x^{q^2}+\beta x^{q^3})^T:x\in\F_{q^4}\right\},\]
with $\alpha\in\F_{q^2}$, $\beta\in\F_{q^4}$ such that $\alpha^{q+1}=1$ and $\beta^{(q^2+1)(q-1)}=1$.
\end{lemma}

\begin{proof}
Note that we can always assume that $U=\{(x,f(x),g(x))^T:x\in\F_{q^4}\}$, with $f,g$ two $\F_q$-linear maps of $\F_{q^4}$. Up to $\mathrm{GL}(3,q^4)$-equivalence, one of the following cases occur

\begin{itemize}
    \item[$1)$] $U=\left\{(x,x^q,x^{q^2})^T:x\in\F_{q^4}\right\}$;
    \item[$2)$] $U_{\alpha}=\left\{(x,x^q+\alpha x^{q^2}, x^{q^3})^T:x\in\F_{q^4}\right\}$, with $\alpha\in\F_{q^4}^*$;
\item[$3)$] $U_{\alpha}=\left\{(x,x^q+\alpha x^{q^3},x^{q^2})^T:x\in\F_{q^4}\right\}$ with $\alpha\in\F_{q^4}^*$;
    \item[$4)$] $U_{\alpha,\beta}=\left\{(x,x^q+\alpha x^{q^3},x^{q^2}+\beta x^{q^3})^T:x\in\F_{q^4}\right\}$ with $\alpha,\beta\in\F_{q^4}^*$;
\end{itemize}

\fbox{Case 1)} (resp. \fbox{Case 2)}). By projecting $L_U$ (resp. $L_{U_\alpha}$) 
from the point $(0:0:1)$ (resp. $(0:1:0)$) to the line with equation $X_2=0$ (resp. $X_1=0$), we obtain the set $$\left\{(x:x^q):x\in\F_{q^4}\right\},$$
(resp. $\{(x:x^{q^3}):x\in\F_{q^4}\}$)
which  is scattered and thus the point $(0:0:1)$ (resp. $(0:1:0)$) 
 is not saturated, since through such a point there does not pass any secant line to $L_U$ (resp. $L_{U_\alpha}$). 

\fbox{Case 3)}
First, by substituting $x$ by $\lambda x$ with $\lambda\in\F_{q^4}^*$ and dividing by $\lambda^q$, the $\F_q$-subspace $U_{\alpha}$ is equivalent to
$$\left\{(x,x^q+\alpha \lambda^{q^3-q} x^{q^3},x^{q^2}):x\in\F_{q^4}\right\}$$ and
since 
$$\F_{q^4}^*=\left\{\alpha z\colon \alpha\in\F_{q^2}^*, z\in\F_{q^4}, z^{q^2+1}=1\right\},$$ we have that $\alpha$ can be chosen in $\F_{q^2}^*$. 

If $\alpha^{q+1}\neq 1$ then by projecting $L_U$ (resp. $L_{U_\alpha}$) 
from the point $(0:0:1)$  to the line with equation $X_2=0$, we obtain the set $$\left\{(x:x^q+\alpha \lambda^{q-q^3} x^{q^3}):x\in\F_{q^4}\right\},$$
which  is scattered and thus the point $(0:0:1)$ is not saturated. 

Consider the case $\alpha = 1$. Note that the point $(1:0:1)$ is of weight $2$. By projecting $L_U$ from the point $(1:0:0)$  to the line with equation $X_0=0$, we obtain the set $$\Lambda :=\left\{(x^q+ x^{q^3}:x^{q^2}):x\in\F_{q^4}\right\}=\left\{(x^q+ x^{q^3}:x):x\in\F_{q^4}\right\}.$$ 
Now, $(0:1)$ is the unique point of $\Lambda $ of weight $2$ and this means that all the lines through $(1:0:0)$ intersect $L_U$ in a unique point and thus $L_U$ is not saturating.

Suppose now $\alpha^{q+1}=1$ and $\alpha \neq 1$. Let $\omega \in \mathbb{F}_{q^2}$ such that $\omega^q+\omega+\alpha\neq 0$. By projecting $L_U$ from the point $(1:0:\omega)$  to the line with equation $X_0=0$, we obtain the set $$\Lambda :=\left\{(x^q+\alpha x^{q^3}:x^{q^2}+\omega x):x\in\F_{q^4}\right\}.$$
The function $x\mapsto x^q+\alpha x^{q^3}$ is a bijection, whose inverse is $x\mapsto \frac{\alpha}{1+\alpha^2}x^{q}+\frac{\alpha^2}{1+\alpha^2} x^{q^3}$. This means that the set $\Lambda$ is equivalent to 
$$\left\{\left(x:\frac{\alpha}{1+\alpha^2}x^{q^3}+\frac{\alpha^2}{1+\alpha^2} x^{q}+\omega \frac{\alpha}{1+\alpha^2}x^{q}+\omega\frac{\alpha^2}{1+\alpha^2} x^{q^3}\right):x\in\F_{q^4}\right\},$$
and so 
$$\Lambda \simeq \left\{(x:(\alpha+1)x^{q^3}+(\omega+\alpha)x^q ):x\in\F_{q^4}\right\},$$
which is scattered if and only if ${\rm N}_{q^4/q}((\omega+\alpha)/(\omega\alpha+1))\neq 1$. 

Since both $\alpha$ and $\omega$ belong to $\mathbb{F}_{q^2}$, the previous condition is equivalent to 
$$\frac{(\omega+\alpha)(\omega^q+1/\alpha)}{(\omega\alpha+1)(\omega^q/\alpha+1)}\neq 1,$$
that is 
$$\omega^q+\omega+\alpha\neq 0.$$
Thus $(1:0:\omega)$ is not saturated.

\fbox{Case 4)} By substituting $x$ by $\lambda x$ with $\lambda\in\F_{q^4}^*$ and dividing by $\lambda^q$, the $\F_q$-subspace $U_{\alpha,\beta}$ is equivalent to
$$\left\{(x,x^q+\alpha \lambda^{q^3-q} x^{q^3},x^{q^2}+\beta\lambda^{q^3-q^2} x^{q^3}):x\in\F_{q^4}\right\}$$ and since 
$$\F_{q^4}^*=\left\{\alpha z\colon \alpha\in\F_{q^2}^*, z\in\F_{q^4}, z^{q^2+1}=1\right\},$$ we have that $\alpha$ can be chosen in $\F_{q^2}^*$. Also, substituting $x$ by $\lambda x$ with $\lambda\in\F_{q^2}^*$ and dividing by $\lambda$, the $\F_q$-subspace $U_{\alpha,\beta}$, with $\alpha\in\F_{q^2}$ is equivalent to $$\left\{(x,x^q+\alpha  x^{q^3},x^{q^2}+\beta\lambda^{q-1} x^{q^3}):x\in\F_{q^4}\right\}$$ and since 
$$\mathbb{F}_{q^4}^*=\left\{\beta z\colon \beta\in\F_{q^4}, z\in\F_{q^2}, \beta^{(q^2+1)(q-1)}=1, z^{q+1}=1 \right\},$$ we have that $\beta$ can be chosen in $\mathbb{F}_{q^4}$ such that $\beta^{(q^2+1)(q-1)}=1$. Finally, by projecting  $L_{U_{\alpha,\beta}}$ from the point $(0:0:1)$ to the line with equation $X_2=0$, we have that the set $$\left\{(x:x^q+\alpha x^{q^3}):x\in\F_{q^4}\right\}$$ is scattered if and only if $\mathrm{N}_{q^4/q}(\alpha) \ne 1$.  This means that if $\alpha\in\F_{q^2}$ and $\alpha^{q+1}\ne 1$, then the point $(0:0:1)$ is not $1$-saturated. Hence we can reduce to the study of the $\F_q$-subspace $U_{\alpha,\beta}$ with $\alpha\in\F_{q^2}$, $\beta\in\F_{q^4}$ such that $\alpha^{q+1}=1$ and $\beta^{(q^2+1)(q-1)}=1$.
\end{proof}

\begin{lemma}\label{lem:first-case}
Let $\alpha^{q+1}=1$, $\beta^{(q^2+1)(q-1)}=1$, with $\alpha \neq 1$. If $q$ is large enough, then $U_{\alpha,\beta}$ is not $2$-saturating.
\end{lemma}

\begin{lemma}\label{Prop:alpha_1 beta_Fq}
Let $\alpha=1$, $\beta\in \mathbb{F}_q^*$. If $q\geq 64$ then $U_{\alpha,\beta}$ is not $2$-saturating.
\end{lemma}

\begin{lemma}\label{lem:second-case}
Let $\alpha=1$, $\beta\notin \mathbb{F}_q$. If $q\geq 64$ then $U_{\alpha,\beta}$ is not $2$-saturating.
\end{lemma}

Theorem \ref{TH} now follows from the above three lemmas, whose proofs are in the Appendix.

\section{Conclusion and open problems}\label{sec:conclusion}

In this paper we continued the investigation of saturating linear sets of small rank introduced in \cite{bonini2022saturating}. In the following remark we resume all the parameters for which we know the exact value of $s_{q^m/q}(k,\rho)$.

\begin{remark}
The lower bound \eqref{eq:lb} is met in the following cases: 
\begin{align*}
    s_{q^m/q}(\rho t,\rho)&=m(t-1)+\rho, & \text{ for all } t\geq 1,m\geq 2,\\
    s_{q^2/q}(k,2)&=k, & \text{ for all } k\geq 2,\\
    s_{q^3/q}(3,2)&=4, & \\
    s_{q^{2r}/q}(3,2) & = r+2, & \text{ for } r\neq 3,5\bmod 6\text{ and }r\geq 4,\\
    s_{q^{2r}/q}(3,2) & = r+2, & \text{ for } \gcd(r,(q^{2s}-q^s+1)!)=1, r\text{ odd}, 1\leq s\leq r, \gcd(r,s)=1,
    \\
    s_{q^{10}/q}(3,2) & = 7, & \text{ for } q=p^{15h+s}, p\in\{2,3\}, \gcd(s,15)=1 \text{ and for } q=5^{15h+1},\\
    s_{q^{10}/q}(3,2) & = 7, & \text{ for } q \text{\ odd}, q= 2,3\bmod 5 \text{ and for } q=2^{2h+1}, h\geq 1,\\
    s_{q^{2r}/q}(3,2) & = r+2, & \text{ for } q \text{\ odd}, q= 2,3\bmod 5 \text{ and for } q=2^{2h+1}, h\geq 1, r\text{ odd},\\
    s_{q^{2r}/q}(2r,2r-1) & =2r+1, & \text{ for all }r\geq 2.
\end{align*}

As shown in Section \ref{sec:non-tightness}, there are some parameters for which \eqref{eq:lb} is not tight:  for $q$  even and large enough, $s_{q^4/q}(3,2)=5>4$.
\end{remark}

The theory of saturating linear sets of small rank needs surely further investigations. We propose some open questions.

\begin{question}
In Section \ref{sec:non-tightness} we showed that $s_{q^4/q}(3,2)=5$ for $q$ even and large enough. Can we prove the same result for $q$ odd? 
\end{question}

The answer to the above question would need to adapt all results and calculations to the odd characteristic. A more ambitious result would be a general answer to the following.

\begin{question}
Is is possible to characterize the parameters for which the bound is not tight?  
\end{question}

A generalized version of Remark \ref{rem:projection} would probably help in finding an answer, at least for some (small) dimensions.

About the constructions, a classification of saturating linear sets meeting the bound would be interesting. A possible step towards such classification would be the following.

\begin{question}
Can we find examples of rank $\rho$-saturating sets in $V(\rho t,q^m)$ of minimal $\F_q$-dimension non-equivalent to those in Theorem \ref{Th:main} and Remark \ref{Remark}? 
\end{question}

\section*{Acknowledgments}
This research was supported by the Italian National Group for Algebraic and Geometric Structures and their Applications (GNSAGA - INdAM). The second author was partially supported by the ANR-21-CE39-0009 - BARRACUDA (French \emph{Agence Nationale de la Recherche}).

\bibliographystyle{abbrv}
\bibliography{Biblio.bib}

\begin{thebibliography}{10}

\bibitem{alfarano2022linear}
G.~N. Alfarano, M.~Borello, A.~Neri, and A.~Ravagnani.
\newblock Linear cutting blocking sets and minimal codes in the rank metric.
\newblock {\em Journal of Combinatorial Theory, Series A}, 192:105658, 2022.

\bibitem{ball2000linear}
S.~Ball, A.~Blokhuis, and M.~Lavrauw.
\newblock Linear {($q$+1)}-fold blocking sets in {PG(2, $q^4$)}.
\newblock {\em Finite Fields and Their Applications}, 6(4):294--301, 2000.

\bibitem{Bartoli:2020aa4}
D.~Bartoli.
\newblock Hasse-weil type theorems and relevant classes of polynomial
  functions.
\newblock {\em London Mathematical Society Lecture Note Series, Proceedings of
  28th British Combinatorial Conference, Cambridge University Press}, to
  appear.

\bibitem{bartoli2021evasive}
D.~Bartoli, B.~Csajb{\'o}k, G.~Marino, and R.~Trombetti.
\newblock Evasive subspaces.
\newblock {\em Journal of Combinatorial Designs}, 29(8):533--551, 2021.

\bibitem{bartoli2018maximum}
D.~Bartoli, M.~Giulietti, G.~Marino, and O.~Polverino.
\newblock Maximum scattered linear sets and complete caps in {G}alois spaces.
\newblock {\em Combinatorica}, 38:255--278, 2018.

\bibitem{blokhuis2000scattered}
A.~Blokhuis and M.~Lavrauw.
\newblock Scattered spaces with respect to a spread in {PG($n$,$q$)}.
\newblock {\em Geometriae Dedicata}, 81:231--243, 2000.

\bibitem{MR0429903}
E.~Bombieri.
\newblock Counting points on curves over finite fields.
\newblock In {\em S\'eminaire Bourbaki : vol. 1972/73, expos\'es 418-435},
  number~15 in S\'eminaire Bourbaki. Springer-Verlag, 1974.
\newblock talk:430.

\bibitem{bonini2022saturating}
M.~Bonini, M.~Borello, and E.~Byrne.
\newblock Saturating systems and the rank covering radius.
\newblock {\em to appear in Journal of Algebraic Combinatorics}, 2023.

\bibitem{byrne2017covering}
E.~Byrne and A.~Ravagnani.
\newblock Covering radius of matrix codes endowed with the rank metric.
\newblock {\em SIAM Journal on Discrete Mathematics}, 31(2):927--944, 2017.

\bibitem{byrne2019assmus}
E.~Byrne and A.~Ravagnani.
\newblock An {A}ssmus--{M}attson theorem for rank metric codes.
\newblock {\em SIAM Journal on Discrete Mathematics}, 33(3):1242--1260, 2019.

\bibitem{MR2206396}
A.~Cafure and G.~Matera.
\newblock Improved explicit estimates on the number of solutions of equations
  over a finite field.
\newblock {\em Finite Fields and their Applications}, 12(2):155--185, 2006.

\bibitem{cohen1997covering}
G.~Cohen, I.~Honkala, S.~Litsyn, and A.~Lobstein.
\newblock {\em Covering codes}.
\newblock Elsevier, 1997.

\bibitem{csajbok2017maximum}
B.~Csajb{\'o}k, G.~Marino, O.~Polverino, and F.~Zullo.
\newblock Maximum scattered linear sets and mrd-codes.
\newblock {\em Journal of Algebraic Combinatorics}, 46:517--531, 2017.

\bibitem{csajbok2021generalising}
B.~Csajb{\'o}k, G.~Marino, O.~Polverino, and F.~Zullo.
\newblock Generalising the scattered property of subspaces.
\newblock {\em Combinatorica}, 41(2):237--262, 2021.

\bibitem{davydov2003saturating}
A.~A. Davydov, S.~Marcugini, and F.~Pambianco.
\newblock On saturating sets in projective spaces.
\newblock {\em Journal of Combinatorial Theory, Series A}, 103(1):1--15, 2003.

\bibitem{davydov2000saturating}
A.~A. Davydov and P.~R. {\"O}sterg{\aa}rd.
\newblock On saturating sets in small projective geometries.
\newblock {\em European Journal of Combinatorics}, 21(5):563--570, 2000.

\bibitem{denaux2022constructing}
L.~Denaux.
\newblock Constructing saturating sets in projective spaces using
  subgeometries.
\newblock {\em Designs, Codes and Cryptography}, 90(9):2113--2144, 2022.

\bibitem{gadouleau2008packing}
M.~Gadouleau and Z.~Yan.
\newblock Packing and covering properties of rank metric codes.
\newblock {\em IEEE Transactions on Information Theory}, 54(9):3873--3883,
  2008.

\bibitem{gadouleau2009bounds}
M.~Gadouleau and Z.~Yan.
\newblock Bounds on covering codes with the rank metric.
\newblock {\em IEEE Communications Letters}, 13(9):691--693, 2009.

\bibitem{MR1988974}
S.~R. Ghorpade and G.~Lachaud.
\newblock \'{E}tale cohomology, {L}efschetz theorems and number of points of
  singular varieties over finite fields.
\newblock {\em Mosc. Math. J.}, 2(3):589--631, 2002.
\newblock Dedicated to Yuri I. Manin on the occasion of his 65th birthday.

\bibitem{MR1962145}
S.~R. Ghorpade and G.~Lachaud.
\newblock Number of solutions of equations over finite fields and a conjecture
  of {L}ang and {W}eil.
\newblock In {\em Number theory and discrete mathematics ({C}handigarh, 2000)},
  Trends Math., pages 269--291. Birkh\"{a}user, Basel, 2002.

\bibitem{gruica2022generalised}
A.~Gruica, A.~Ravagnani, J.~Sheekey, and F.~Zullo.
\newblock Generalised scattered subspaces.
\newblock {\em arXiv preprint arXiv:2207.01027}, 2022.

\bibitem{MR0463157}
R.~Hartshorne.
\newblock {\em Algebraic geometry}.
\newblock Springer-Verlag, New York-Heidelberg, 1977.
\newblock Graduate Texts in Mathematics, No. 52.

\bibitem{MR65218}
S.~Lang and A.~Weil.
\newblock Number of points of varieties in finite fields.
\newblock {\em American Journal of Mathematics}, 76:819--827, 1954.

\bibitem{lia2023short}
S.~Lia, G.~Longobardi, G.~Marino, and R.~Trombetti.
\newblock Short rank-metric codes and scattered subspaces.
\newblock {\em arXiv preprint arXiv:2306.01315}, 2023.

\bibitem{MR1429394}
R.~Lidl and H.~Niederreiter.
\newblock {\em Finite fields}, volume~20 of {\em Encyclopedia of Mathematics
  and its Applications}.
\newblock Cambridge University Press, Cambridge, second edition, 1997.
\newblock With a foreword by P. M. Cohn.

\bibitem{lunardon1999normal}
G.~Lunardon.
\newblock Normal spreads.
\newblock {\em Geom. Dedicata}, 75(3):245--261, 1999.

\bibitem{nagy2018saturating}
Z.~L. Nagy.
\newblock Saturating sets in projective planes and hypergraph covers.
\newblock {\em Discrete Mathematics}, 341(4):1078--1083, 2018.

\bibitem{napolitano2021linear}
V.~Napolitano, O.~Polverino, G.~Zini, and F.~Zullo.
\newblock Linear sets from projection of desarguesian spreads.
\newblock {\em Finite Fields and Their Applications}, 71:101798, 2021.

\bibitem{pambianco1996small}
F.~Pambianco and L.~Storme.
\newblock Small complete caps in spaces of even characteristic.
\newblock {\em Journal of Combinatorial Theory, Series A}, 75(1):70--84, 1996.

\bibitem{polverino2010linear}
O.~Polverino.
\newblock Linear sets in finite projective spaces.
\newblock {\em Discrete Mathematics}, 310(22):3096--3107, 2010.

\bibitem{MR2121285}
W.~Schmidt.
\newblock {\em Equations over finite fields: an elementary approach}.
\newblock Kendrick Press, Heber City, UT, second edition, 2004.

\bibitem{ughi1987saturated}
E.~Ughi.
\newblock Saturated configurations of points in projective {G}alois spaces.
\newblock {\em European Journal of Combinatorics}, 8(3):325--334, 1987.

\end{thebibliography}

\newpage

\appendix

\section{Appendix}
In this Appendix we will make use of algebraic varieties over finite fields; see \cite{Bartoli:2020aa4} for a survey on links between algebraic varieties over finite fields and relevant combinatorial objects. 

A variety and more specifically a curve, i.e. a variety of dimension 1, are described by a certain set of equations with coefficients  in a finite field $\mathbb{F}_q$. A variety  defined by a unique equation is called a hypersurface. We say that a variety $\mathcal{V}$ is \emph{absolutely irreducible} if there are no varieties $\mathcal{V}^{\prime}$ and $\mathcal{V}^{\prime\prime}$ defined over the algebraic closure of $\mathbb{F}_q$ and different from $\mathcal{V}$ such that $\mathcal{V}= \mathcal{V}^{\prime} \cup \mathcal{V}^{\prime\prime}$. If a variety $\mathcal{V}\subseteq \PG(k-1,q)$ is defined by $F_i(X_0,\ldots, X_k)=0$, for  $i\in [s]$, an $\mathbb{F}_{q}$-rational point of $\mathcal{V}$ is a point $(x_0:\ldots:x_k) \in \PG(k-1,q)$ such that $F_i(x_0,\ldots, x_k)=0$, for  $i\in [s]$. A point is affine if $x_0\neq 0$. The set of the $\mathbb{F}_q$-rational points of $\mathcal{V}$ is usually denoted by $\mathcal{V}(\mathbb{F}_q)$.

In what follows, we mainly focus on algebraic hypersurfaces, i.e. algebraic varieties that may be defined by a single implicit equation. An algebraic hypersurface defined over a finite field $\mathbb{F}_q$  is \emph{absolutely irreducible}  if the associated polynomial is irreducible over every algebraic extension of $\mathbb{F}_q$. An absolutely irreducible $\mathbb{F}_q$-rational component of a hypersurface $\mathcal{S}$, defined by the polynomial $F$, is simply an absolutely irreducible hypersurface such that the associated polynomial has coefficients in $\mathbb{F}_q$ and it is a factor of $F$.   For a deeper introduction to algebraic varieties we refer the interested reader to \cite{MR0463157}.

In the small-degree regime (usually when $\max\{\deg(f),\deg(g)\}\lesssim \sqrt[4]{q}$), the existence of an absolutely irreducible component in a hypersurface (or more in general a variety) yields the existence of suitable $\mathbb{F}_q$-rational points of the hypersurface itself, due to estimates on the number of  $\mathbb{F}_{q}$-rational points of an algebraic variety such as the Lang-Weil bound  \cite{MR65218} and its genera\-lizations.

\begin{theorem}[Lang-Weil Theorem]\label{Th:LangWeil}
Let $\mathcal{V}\subseteq \PG(k-1,q)$ be an absolutely irreducible variety of dimension $n$ and degree $d$. Then there exists a constant $C$ depending only on $k$, $n$, and $d$ such that 
\begin{equation*}
\left|\#(\mathcal{V}(\mathbb{F}_q))-\sum_{i=0}^{n} q^i\right|\leq (d-1)(d-2)q^{n-1/2}+Cq^{n-1}.
\end{equation*}
\end{theorem}

 Although the constant $C$ was not  computed in \cite{MR65218}, explicit estimates have been provided for instance in  \cite{MR0429903,MR2206396,MR1988974,MR1962145,MR1429394,MR2121285} and they have the general shape $C=r(d)$ provided that $q>s(n,d)$, where $r$ and $s$ are polynomials of (usually) small degree. We refer to \cite{MR2206396} for a survey on these bounds. In the following we will make use of Theorem \ref{Th:LangWeil} in small  degree regime. Actually, for a variety having an absolutely irreducible component defined over $\mathbb{F}_q$, when the the degree of the variety is nondepending on $q$ and for $q$ large enough, Theorem \ref{Th:LangWeil} yields the existence of roughly $q^n-O(q^{n-1/2})$ $\mathbb{F}_q$-rational points.

Consider

\begin{eqnarray*}
d_{\alpha,\beta}(C) &:=&((\alpha^4+1)C^4 + \alpha(\alpha+1)^2{\rm Tr}_{q^4/q}(\beta)C^3 + (\alpha^4 \beta^{q^3+q} + \alpha^2 {\rm Tr}_{q^4/q}(\beta^{q+1}) +\beta^{q^2+1})C^2\\
&&+\alpha(\alpha^2 (\beta^{q^3+q+1} +\beta^{q^3+q^2+q}) + \alpha {\rm Tr}_{q^4/q}(\beta) + \beta^{q^2+q+1}+\beta^{q^3+q^2+1})C+\alpha^2({\rm N}_{q^4/q}(\beta)+1))^2,
\end{eqnarray*}
 
\begin{eqnarray*}
f_{\alpha,\beta}(C) &:=& 
\alpha^3\beta^{q^3} C^{q^2+q+1}+ \alpha \beta^{q^2} C^{q^3+q+1}+ \alpha\beta C^{q^3+q^2+q}+   \alpha^3\beta^q C^{q^3+q^2+1}
+ (\alpha^3 + \alpha^2\beta^{q^3+q^2}) C^{q+1}\\
&&+ \alpha^4\beta^{q^3+q} C^{q^2+1} + 
(\alpha^2\beta^{q^2+q} + \alpha)C^{q^3+1}
+ (\alpha^3 + \alpha^2\beta^{q+1} C^{q^3+q^2}+ 
(\alpha^2\beta^{q^3+1}+ \alpha )C^{q^2+q}\\
&&+\beta^{q^2+1}C^{q^3+q} + (\alpha^3\beta^{q^3+q^2+q} + \alpha^2\beta^{q^3})C + 
(\alpha^2\beta + \alpha\beta^{q^3+q^2+1})C^q
+(\alpha^3 \beta^{q^3+q+1} + \alpha^2\beta^q)C^{q^2}\\
&&+(\alpha^2\beta^{q^2} + \alpha\beta^{q^2+q+1})C^{q^3} + \alpha^2 {\rm N}_{q^4/q}(\beta),
\end{eqnarray*}
and 
\begin{eqnarray*}
g_{\alpha,\beta}(C) &:=& {\rm N}_{q^4/q}(\alpha C^{q+1} + \alpha C^{q^3+1}+ \alpha^2\beta^q C + \beta C^q + \alpha\beta^{q+1})\\
    &&\Big((\alpha+1)C^{q^3+q^2+q+1}  + (\alpha^3+ 
        \alpha^2 + \alpha)\beta^{q^3} C^{q^2+q+1} + \alpha \beta^{q^2}  C^{q^3+q+1} + \\
        &&+( \alpha^3 + \alpha^2 \beta^{q^3+q^2})C^{q+1}+ \beta^q(\alpha^3  + \alpha^2 + \alpha ) C^{q^3+q^2+1} + 
        (\alpha^4 +\alpha^3+\alpha^2) \beta^{q^3+q})C^{q^2+1}\\
        && + (\beta^{q^2+q} +1)\alpha^2 C^{q^3+1}+ \alpha^3( beta^{q^3+q^2+q} + \beta^{q^3} )C+ 
        \alpha \beta C^{q^3+q^2+q}\\
        &&+\alpha^2 ( \beta^{q^3+1} + 1)C^{q^2+q} + 
        \beta^{q^2+1}C^{q^3+q}+ (\alpha^2 \beta + \alpha \beta^{q^3+q^2+1})C^q \\
        &&+ (\alpha^3 +\alpha^2 \beta^{q+1})C^{q^3+q^2}+ \alpha^3 (\beta^{q^3+q+1} +  \beta^q)C^{q^2} + (\alpha^2 \beta^{q^2} +\\
        &&\alpha \beta^{q^2+q+1})C^{q^3} + \alpha^3 + \alpha^2 \beta^{q^3+q^2+q+1}+ \alpha^2\Big)\cdot\\
    &&\Big((\alpha^4 + \alpha^2)C^{q^3+q^2+q+1} +\alpha^3 {\rm Tr}_{q^4/q}(\beta^{q^3} C^{q^2+q+1})+ (\alpha^2 \beta^{q^2+q^3} + \alpha) C^{q+1} \\
        &&+ 
        \alpha^4 \beta^{q^3+q} C^{q^2+1}+  \alpha^2 \beta^{q^2+1} C^{q^3+q}
        (\alpha^2 \beta^{q^2+q} + \alpha )C^{q^3+1} + 
        (\alpha^2 \beta^{q^3+1} + \alpha) C^{q^2+q}\\
        &&+ (\alpha^2 \beta^{q+1}+ \alpha )C^{q^3+q^2}+
        (\alpha^3 \beta^{q^3+q^2+q}+ \alpha^2 \beta^{q^3}) C + 
        (\alpha \beta^{q^3+q^2+1} +  \beta )C^q\\
        &&+  (\alpha^3 \beta^{q^3+q+1} + \alpha^2 \beta^q)C^{q^2} + (\alpha \beta^{q^2+q+1} + \beta^{q^2})C^{q^3} + 
        \alpha^2 \beta^{q^3+q^2+q+1} + \alpha^2 + 1\Big).
\end{eqnarray*}
        
Let $$\Gamma_{\alpha,\beta} := \{c \in \mathbb{F}_q^4: f_{\alpha,\beta}(c)=0, \quad (c^2\alpha^2 + \beta^2) (c^{2q+2} \alpha^2 + c^{2q} \beta^2 + 1\
)d_{\alpha,\beta}(c)g_{\alpha,\beta}(c)\neq0\}.$$

\begin{lemma}\label{Prop:Gamma}
Let $\alpha^{q+1}=1$, $\beta^{(q^2+1)(q-1)}=1$, with $\alpha \neq 1$. If $q$ is large enough then $\Gamma_{\alpha,\beta}\neq \emptyset$.
\end{lemma}
\begin{proof}
First observe that the polynomials $f_{\alpha,\beta}(C)$, $d_{\alpha,\beta}(C)$, $g_{\alpha,\beta}(C)$ are nonvanishing.

To prove that $\Gamma_{\alpha,\beta}$ is not empty, we will make use the following approach. Set $X:=C$, $Y:= C^q$, $Z:=C^{q^2}$, $T:=C^{q^3}$, let $\{\xi,\xi^q,\xi^{q^2},\xi^{q^3}\}$ be a normal basis of $\mathbb{F}_{q^4}$ over $\mathbb{F}_q$, and write $C=C_0\xi +C_1\xi^q+C_2\xi^{q^2}+C_3\xi^{q^3}$, where $C_0,C_1,C_2,C_3\in \mathbb{F}_q$. Also, let $\alpha=\alpha_0\xi +\alpha_1\xi^q+\alpha_2\xi^{q^2}+\alpha_3\xi^{q^3}$ and $\beta=\beta_0\xi +\beta_1\xi^q+\beta_2\xi^{q^2}+\beta_3\xi^{q^3}$. 

Write $f_{\alpha,\beta}(C)$ as $f_{\alpha,\beta}(X,Y,Z,T)$. Since $\alpha^{q+1}=1$, $(f_{\alpha,\beta}(C))^q=f_{\alpha,\beta}(C)$ and thus the hypersurface defined by $\mathcal{Y} : f_{\alpha,\beta}(C_0\xi +C_1\xi^q+C_2\xi^{q^2}+C_3\xi^{q^3})=0$ is $\mathbb{F}_q$-rational and projectively $\mathbb{F}_{q^4}$-isomorphic to the hypersurface  $\mathcal{X} : f_{\alpha,\beta}(X,Y,Z,T)=0$ and of degree $3$. 

Note that  $\mathcal{X}$ is of degree one in $T$ and thus is reducible only if $f_{\alpha,\beta}(X,Y,Z,T)= \ell_1(X,Y,Z) T+\ell_2(X,Y,Z)$ possesses a factor depending only on $X,Y,Z$, which should be a common factor of the coefficient $\ell_1(X,Y,Z)$ and $\ell_{2}(X,Y,Z)$. That is to say, the resultant between $\ell_1(X,Y,Z)$ and $\ell_{2}(X,Y,Z)$ with respect to $Y$ should vanish. This is not possible since $\alpha\beta\neq0$, as easy computations show. 

The argument above yields the absolutely irreducibility of $\mathcal{X}$ and thus of $\mathcal{Y}$, since the absolutely irreducibility is preserved by projectivity. Because $\mathcal{Y}$ is defined over $\mathbb{F}_q$ and absolutely irreducible, it has dimension $3$ and it contains roughly at least $q^3-O(q^{5/2})$ $\mathbb{F}_q$-rational points $(x_0,y_0,z_0,t_0)$ which correspond to values $C=x_0 \xi+y_0\xi^q+x_2\xi^{q^2}+x_3\xi^{q^3}$ satisfying $f_{\alpha,\beta}(C)=0$.

In order to show the existence of values $C\in\mathbb{F}_{q^4}$ in $\Gamma_{\alpha,\beta}$ it is enough to show that the hypersurface 
$$\mathcal{Z} : (X^2\alpha^2 + \beta^2) (X^2Y^2 \alpha^2 +  \beta^2Y^2 + 1\
)d_{\alpha,\beta}(X,Y,Z,T)g_{\alpha,\beta}(X,Y,Z,T)=0$$
does not contain $\mathcal{X}$, so that $\mathcal{Z}\cap \mathcal{X}$ is subvariety of codimension one in $\mathcal{X}$. We have only to check that this actually holds for the hypersurface $g_{\alpha,\beta}(X,Y,Z,T)=0$, being trivial for the other components of $\mathcal{Z}$. The four factors of degree 3 in $g_{\alpha,\beta}(X,Y,Z,T)$ cannot be factors of $f_{\alpha,\beta}(X,Y,Z,T)$, being this polynomial irriducible. It can be easily checked that the resultants of each of two other factors of $g_{\alpha,\beta}(X,Y,Z,T)$ and $f_{\alpha,\beta}(X,Y,Z,T)$ with respect to $T$ are nonvanishing polynomials and thus $\mathcal{X}$ and $\mathcal{Z}$ do not share any component. Thus  $\mathcal{Z}\cap \mathcal{X}$ is of dimension $2$ and there are at most $O(q^2)$ $\mathbb{F}_q$-rational points $(x_0,y_0,z_0,t_0)$ whose corresponding value $C=x_0 \xi+y_0\xi^q+x_2\xi^{q^2}+x_3\xi^{q^3}$ satisfies  $f_{\alpha,\beta}(C)=0$ and $ (C^2\alpha^2 + \beta^2) (C^{2q+2} \alpha^2 + C^{2q} \beta^2 + 1\
)d_{\alpha,\beta}(C)g_{\alpha,\beta}(C)=0$. This yields $|\Gamma_{\alpha,\beta}|=q^3-O(q^{5/2})$ and the claim follows.
\end{proof}

\begin{proof}[Proof of Lemma \ref{lem:first-case}]
We will prove that  each point $P_C:=(0:1:C)\in \mathrm{PG}(2,q^4)$, $C\in \Gamma_{\alpha,\beta}$, is not $1$-saturated by $L_{U_{\alpha,\beta}}$. Since by Lemma \ref{Prop:Gamma} $\Gamma_{\alpha,\beta}\neq 0$, this will prove the claim.

Since $U_{\alpha,\beta}$ is $2$-saturating if and only if $U_{\alpha^2,\beta^2}$ is $2$-saturating, in the following, for the seek of convenience we will consider this latter case.  

The point $P_C$ does not belong to $L_{U_{\alpha^2,\beta^2}}$ for any $C\in \mathbb{F}_{q^2}$.

Recall that $P_C$ is saturated if and only if 
$$
\det\begin{pmatrix}
x&x^q+\alpha^2 x^{q^3}&x^{q^2}+\beta^2 x^{q^3}\\
y&y^q+\alpha^2 y^{q^3}&y^{q^2}+\beta^2 y^{q^3}\\
0&1&C
\end{pmatrix}=0
$$
for some $x,y \in \mathbb{F}_{q^4}$ with $xy^q-x^qy\neq0$. By direct checking, the above determinant reads 
$$F_0(x,y) := C^2 x y^q  + xy^{q^2} + C^2\alpha^2 x y^{q^3}  + \beta^2 x y^{q^3}  + C^2 x^q y + x^{q^2} y + C^2\alpha^2 x^{q^3} y  + \beta^2 x^{q^3} y.$$
Denote by $F_i(x,y):=(F_0(x,y))^{q^i}$, $i=1,2,3$. 

Let $G_1(x,y),G_2(x,y),G_3(x,y)$ be the polynomials
\begin{eqnarray*}
G_1(x,y)&:=&(C^2y^q + y^{q^2} + C^2\alpha^2 y^{q^3} + \beta^2y^{q^3}) F_1(x,y)+(C^{2q}y^q  + \alpha^2\beta^{2q} y^q)F_0(x,y),\\
G_2(x,y)&:=&(C^2y^q + y^{q^2} + C^2\alpha^2 y^{q^3} + \beta^2y^{q^3}) F_2(x,y)+ y^{q^2} F_0(x,y),\\
G_3(x,y)&:=&(C^2y^q + y^{q^2} + C^2\alpha^2 y^{q^3} + \beta^2y^{q^3}) F_3(x,y)+ C^{2q^3}\alpha^2 y^{q^3} F_0(x,y).\\
\end{eqnarray*}
One can check that $G_i$ are $q$-linearized polynomials in $x$ containing only $x^q,x^{q^2},x^{q^3}$. We consider now the polynomials $H_1$ and $H_2$  
\begin{eqnarray*}
H_1(x,y)&:=&u G_1(x,y)+vG_2(x,y),\\
H_2(x,y)&:=&w G_1(x,y)+vG_3(x,y),\\
\end{eqnarray*}
where

\begin{eqnarray*}
u&:=&y^{q^2}(C^2 y + (\alpha^2 C^{2q^2+2}+\beta^{2q^2}C^2)y^q +
(\alpha^2 C^{2q^2}+\beta^{2q^2})y^{q^2}\\
&&+(\alpha^4 C^{2q^2}+\alpha^2\beta^{2q^2}C^2+\alpha^2\beta^2C^{2q^2}+\beta^{2q^2+2})y^{q^3})\\
v&:=&(C^{2q} + \alpha^2\beta^{2q}) y^{q^2+1} + (C^{2+2q}2\alpha^2 + 
        C^2\alpha^4\beta^{2q}+C^{2q}\beta^2+\alpha^2\beta^{2q+2})y^{q^3+1}+ C^{2q+2}\alpha^2 y^{q^2+q} \\
        &&  + 
        C^2\alpha^2 y^{q^3+q} + C^{2q}\alpha^2 y^{2q^2} + (C^{2q+2}\alpha^4 + 
        C^{2q}\alpha^2\beta^2 + \alpha^2) y^{q^3+q^2}+ (C^2\alpha^4 + 
        \alpha^2\beta^2)y^{2q^3},\\
w&:=& \alpha^2 (C^{2q^3+2} y + C^2 y^q + y^{q^2} + C^2 \alpha^2 y^{q^3} + \beta^2 y^{q^3})y^{q^3}.
\end{eqnarray*}

Now,
$$H_1(x,y)=(C^2y^q + y^{q^2} + C^2\alpha^2 y^{q^3} + \beta^2y^{q^3})(x y^q + x^q y)^{q^2} L_1(y) 
$$
and 
$$H_2(x,y)=(C^2y^q + y^{q^2} + C^2\alpha^2 y^{q^3} + \beta^2y^{q^3})(x y^q + x^q y)^{q^2} L_2(y), 
$$
where
\begin{eqnarray*}
L_1(y)&:=&
(C^{2q+2}\alpha^2 + C^2\alpha^4\beta^{2q} + C^{2q}\beta^2 + 
    \alpha^2\beta^{2q+2})y^2+ (C^{2q^2+2q+2}\alpha^4 + 
    C^{2q+2}\alpha^2\beta^{2q^2}\\&& + C^{2q^2+2}\alpha^6\beta^{2q} + 
    C^2\alpha^4\beta^{2q^2+2q} + C^2\alpha^2+C^{2q^2+2q}\alpha^2\beta^2+C^{2q}\beta^{2q^2+2} + 
    C^{2q^2}\alpha^4\beta^{2q+2}\\&& + \alpha^2\beta^{2q^2+2q+2})y^{q+1} + 
    (C^{2q+2}\alpha^4 + C{2q^2+2} + C^{2q}\alpha^2\beta^2 + 
    C^{2q^2}\alpha^2\beta^{2q})y^{q^2+1}\\&&
    +( C^{2q^2+2q+2}\alpha^2 + 
     C^{2q^2+2}\alpha^4\beta^{2q} +  C^2\alpha^4 + 
     C^{2q^2+2q}\beta^2 +  C^{2q^2}\alpha^2\beta^{2q+2} + 
     \alpha^2 \beta^2)y^{q^3+1}\\&&
    + (C^{2q+2}\alpha^4 + C^2\alpha^2\beta^{2q^2})y^{2q} +
    (C^{2q^2+2q+2}\alpha^2 + C^{2q^2}\alpha^4 + 
    \alpha^2\beta^{2q^2})y^{q^2+q}\\&&
    +(C^{2q^2+2}\alpha^6 + 
    C^{2q^2+2}\alpha^2 + C^2\alpha^4\beta^{2q^2} + 
    C^{2q^2}\alpha^4\beta^2 + \alpha^2\beta^{2q^2+2})y^{q^3+q}\\&&
    + C^{2q^2+2q}\alpha^2 y^{2q^2}+ 
    (C^{2q^2+2q+2}\alpha^4 + C^{2q^2+2q}\alpha^2\beta^2 + C^{2q^2}\alpha^2)y^{q^3+q^2}\\&&
    + (C^{2q^2+2}\alpha^4 + C^{2q^2}\alpha^2\beta^2 ) y^{2q^3}       
\end{eqnarray*}
and 
\begin{eqnarray*}
L_2(y)&:=&(C^{2q^3+2q^2}\alpha^2 + C^{2q^3}\alpha^4\beta^{2q})y^2
+ (C^{2q^3+2q^2+2}\alpha^4
    + C^{2q}\alpha^2 + \alpha^4\beta^{2q})y^{q+1}\\&&
    + (C^{2q^3+2q}\alpha^4 + 
     C^{2q^3+2q} + C^{2q}\alpha^2\beta^{2q^3}+ C^{2q^3}\alpha^2\beta^{2q}
    + \alpha^4\beta^{2q^3+2q})y^{q^2+1}\\&&
    + (  C^{2q^3+2q+2} \alpha^2 + 
      C^{2q+2} \alpha^4 \beta^{2q^3} +   C^{2q^3+2} \alpha^4 \beta^{2q} + 
     C^2 \alpha^6 \beta^{2q^3+2q} +   C^{2q^3+2q}\beta^2 +\\&& 
      C^{2q} \alpha^2 \beta^{2q^3+2} +   C^{2q^3} \alpha^4 + 
      C^{2q^3} \alpha^2 \beta^{2q+2} +  \alpha^4 \beta^{2q^3+2q+2})y^{q^3+1}
    +C^{2q+2} \alpha^4 y^{2q}\\&&
    + (C^{2q^3+2q+2} \alpha^2 + 
    C^{2q+2} \alpha^4 \beta^{2q^3} + C^{2q} \alpha^4 )y^{q^2+q}
    + (C^{2q+2} \alpha^6 + C^{2q^3+2} \alpha^2 + 
    C^2 \alpha^4 \beta^{2q^3}\\&& + C^{2q} \alpha^4\beta^2 )y^{q^3+q^2}
    + (C^{2q^3+2q} \alpha^2 + C^{2q} \alpha^4 \beta^{2q^3} )y^{q^2}
    +(C^{2q^3+2q+2} \alpha^4 + C^{2q+2} \alpha^6 \beta^{2q^3}\\&& + 
    C^{2q^3+2q} \alpha^2\beta^2 + C^{2q} \alpha^4 \beta^{2q^3+2} + 
    C^{2q^3} \alpha^2 + \alpha^4 \beta^{2q^3} )y^{q^3+q^2}\\&&
    + (C^{2q^3+2} \alpha^4 + C^2 \alpha^6 \beta^{2q^3} + C^{2q^3} \alpha^2\beta^2 + 
    \alpha^4 \beta^{2q^3+2})y^{2q^3}.
\end{eqnarray*}
Note that the determinant of the Dickson matrix of $C^2y^q + y^{q^2} + C^2\alpha^2 y^{q^3} + \beta^2y^{q^3}$ is $d_{\alpha,\beta}(C)$ and thus $d_{\alpha,\beta}(C)\neq 0$ for any $C \in \Gamma_{\alpha,\beta}$.
       
        We already excluded the case $x y^q + x^q y=0$ and thus from $H_1(x,y)=H_2(x,y)=0$ one gets $L_1(y)=L_2(y)=0$. 
Let $P(y):= \alpha^2C^{2q^3}(C^q + \alpha \beta^q)^2L_1(y)+(C^q+\alpha\beta^q)^2(C\alpha + \beta)^2 L_2(y)=(C^{q} + \alpha\beta^q)^2Q(y)$.
From $Q(y)=0$ one gets  $y=U/V$, where
\begin{eqnarray*}
U &:=&(\alpha^2(C^2y^q + y^{q^2} + C^2 \alpha^2 y^{q^3}+ y^{q^3}\beta^2)\cdot \\&&
\Big( (C^{2q+2} \alpha^4 +  C^{2q} \alpha^2\beta^2  +  C^{2q^3+2q^2} \alpha^4 + 
        C^{2q^3} \alpha^2 \beta^{2q^2}) y^{q}\\&& + (C^{2q^3+2q+2} \alpha^2 + 
        C^{2q+2} \alpha^4 \beta^{2q^3} +   C^{2q^3+2q^2+2q} \alpha^2 + 
        C^{2q^3+2q}\beta^2 + C^{2q} \alpha^2 \beta^{2q^3+2})y^{q^2}\\&& + 
        (C^{2q^3+2} \alpha^2 + C^2 \alpha^4 \beta^{2q^3} +  C^{2q^3+2q^2} \alpha^2 + 
        C^{2q^3}\beta^2 + \alpha^2 \beta^{2q^3+2})y^{q^3}\Big),
\end{eqnarray*}
\begin{eqnarray*}
V&:=& (C^4  C^{2q^3+2q} \alpha^6 +  C^{2q+2}  C^{2q^3+2q^2} \alpha^6 + 
         C^{2q^3+2q+2} \alpha^4\beta^2 +  C^{2q^3+2q+2} \alpha^4 \beta^{2q^2} +
         C^{2q+2} \alpha^4\\&&+  C^{2q^3+2q^2+2} \alpha^8 \beta^{2q} + 
         C^{2q^3+2} \alpha^6 \beta^{2q} \beta^{2q^2} +  C^{2q^3+2} \alpha^4 + 
         C^2 \alpha^6 \beta^{2q} +  C^{2q}  C^{2q^3+2q^2} \alpha^4\beta^2\\&& + 
         C^{2q^3+2q} \alpha^2\beta^2 \beta^{2q^2} +  C^{2q} \alpha^2\beta^2 + 
         C^{2q^3+2q^2} \alpha^6 \beta^{2q+2} + 
         C^{2q^3} \alpha^4 \beta^{2q+2} \beta^{2q^2} +  \alpha^4 \beta^{2q+2})y^q\\&& + 
        (C^{2q^3+2q+2} 
        \alpha^2 +    C^{2q+2} \alpha^4 \beta^{2q^3} + 
           C^{2q^3+2} \alpha^4 \beta^{2q} +   C^2 \alpha^6 \beta^{2q^3+2q} + 
           C^{2q}  C^{2q^3+2q^2} \alpha^2\\&& +    C^{2q^3+2q}\beta^2 + 
           C^{2q} \alpha^2 \beta^{2q^3+2} +    C^{2q^3+2q^2} \alpha^4 \beta^{2q} + 
           C^{2q^3} \alpha^2 \beta^{2q+2} +   \alpha^4 \beta^{2q^3+2q+2})y^{q^2}\\&& + 
        (C^4  C^{2q^3+2q} \alpha^4 +   C^4  C^{2q} \alpha^6 \beta^{2q^3} + 
          C^4  C^{2q^3} \alpha^6 \beta^{2q} +   C^4 \alpha^8 \beta^{2q^3+2q} + 
           C^{2q+2}  C^{2q^3+2q^2} \alpha^4\\&& +    C^{2q^3+2q^2+2} \alpha^6 \beta^{2q} + 
           C^{2q}  C^{2q^3+2q^2} \alpha^2\beta^2 +    C^{2q^3+2q}\beta^4 + 
           C^{2q} \alpha^2\beta^4 \beta^{2q^3} +    C^{2q^3+2q^2} \alpha^4 \beta^{2q+2}\\&& + 
           C^{2q^3} \alpha^2\beta^4 \beta^{2q} +   \alpha^4\beta^4 \beta^{2q^3+2q})y^{q^3}.
\end{eqnarray*}

Substituting it in $L_1(y)=0$ we obtain 
\begin{equation}\label{eq:L}
    \alpha^2(C^2\alpha^2 + \beta^2) (C^{2q+2} \alpha^2 + C^{2q} \beta^2 + 1\
)(C^2y^q+ y^{q^2} + C^2 \alpha^2 y^{q^3} + \beta^2 y^{q^3}) y^qM(y)=0,
\end{equation}
where 
$$M(y)=a_{00}^2y^{2q^3}+a_{01}^2y^{q^3+q^2}+a_{10}^2y^{q^3+q}+a_{02}^2y^{2q^2}+a_{11}^2y^{q^2+q}+a_{20}^2y^{2q},$$
with
\begin{eqnarray*}
a_{00}&:=& \alpha(C^{q^2+q} + C^{q^2} \alpha \beta^{q} + \alpha)(C^{q^3+1} \alpha + C \alpha^2 \beta^{q^3} + C^{q^3+q^2} \alpha + C^{q^3}\beta + \alpha \beta^{q^3+1})\\&&
\cdot(C^{q+1} \alpha + C^{q^3+1} \alpha + C \alpha^2 \beta^{q} + C^{q}\beta + \alpha \beta^{q+1}),\\
a_{01}&:=&
   (C^{q} + \alpha \beta^{q})
   (C^{q^3+1} \alpha + C \alpha^2 \beta^{q^3} + C^{q^3+q^2} \alpha + C^{q^3}\beta + \alpha \beta^{q^3+1})
   f_{\alpha,\beta}(C),\\
a_{02} &:=&
     (C^{q^3+q} \alpha^2 + C^{q^3+q} + C^{q} \alpha \beta^{q^3} + C^{q^3} \alpha \beta^{q} + \alpha^2 \beta^{q^3+q})\\&&
    \cdot(C^{q^3+1} \alpha + C \alpha^2 \beta^{q^3} + C^{q^3+q^2} \alpha + C^{q^3}\beta + \alpha \beta^{q^3+1})\\&&
    \cdot(C^{q+1} \alpha + C^{q^2+q} \alpha + C^{q} \beta^{q^2} + C^{q^2} \alpha^2 \beta^{q} + \alpha \beta^{q^2+q}),\\
a_{10} &:=&
    \alpha
    (C^{q+1} \alpha + C^{q^3+1} \alpha + C \alpha^2 \beta^{q} + C^{q}\beta + \alpha \beta^{q+1})
     f_{\alpha,\beta}(C),\\
a_{11} &:=&
    \alpha C^{q^3}
    (C^{q+1} \alpha + C^{q^2+q} \alpha + C^{q} \beta^{q^2} + C^{q^2} \alpha^2 \beta^{q} + \alpha \beta^{q^2+q})
     f_{\alpha,\beta}(C),\\
a_{20} &:=&
    \alpha^2
    (C^{q^3+q^2} \alpha + C^{q^3} \beta^{q^2} + 1)
    (C^{q+1} \alpha + C^{q^2+q} \alpha + C^{q} \beta^{q^2} + C^{q^2} \alpha^2 \beta^{q} + \alpha \beta^{q^2+q})\\&&
    (C^{q+1} \alpha + C^{q^3+1} \alpha + C \alpha^2 \beta^{q} + C^{q}\beta + \alpha \beta^{q+1}).
\end{eqnarray*}
Since $C\in \Gamma_{\alpha,\beta}$ and $\alpha y\neq0$,  \eqref{eq:L} yields $M(y)=0$. Also, from $f_{\alpha,\beta}(C)=0$, $M(y)=({a_{00}}y^{q^3}+{a_{02}}y^{q^2}+{a_{20}}y^{q})^2$. 
We will show that $M(y)$ has only the zero root in $\mathbb{F}_{q^4}$. To this aim consider the Dickson matrix associated with $M(y)$ 
$$
\begin{pmatrix}
0&{a_{20}}&{a_{02}}&{a_{00}}\\
{a_{00}}^q&0&{a_{20}}^q&{a_{02}}^q\\
{a_{02}}^{q^2}&{a_{00}}^{q^2}&0&{a_{20}}^{q^2}\\
{a_{20}}^{q^3}&{a_{02}}^{q^3}&{a_{00}}^{q^3}&0\\
\end{pmatrix},
$$
whose determinant is precisely  $g_{\alpha,\beta}(C)$, which is nonvanishing since $C\in \Gamma_{\alpha,\beta}$. 

This shows that $P_C$ is not saturated by $L_{U_{\alpha,\beta}}$ and the claim follows.
\end{proof}

Let
\begin{eqnarray*}
r_0(z)&:= (z^4-z)(z-\beta)(z^2 + z\beta^2 + z\beta + 1)(z^2 + z\beta^3 + z\beta^2 + z\beta + \beta^3 + \beta^2 + 1)(z\beta^3+\beta+1),
\end{eqnarray*}
and 
$$\Delta_0 :=\left\{z \in \mathbb{F}_q \ : \ {\rm Tr}_{q/2}\left(\frac{(z\beta^3 + z\beta^2 + z + \beta)(z+\beta)}{z^2\beta^4}\right)=0, r_0(z)\neq 0  \right\}.$$

\begin{lemma}
Let $\beta\in \mathbb{F}_q^*$. If $q\geq 64$ then $\Delta_0\neq \emptyset$.
\end{lemma}
\begin{proof}
Recall that ${\rm Tr}_{q/2}\left(\frac{(z\beta^3 + z\beta^2 + z + \beta)(z+\beta)}{z^2\beta^4}\right)=0$ if and only if there exists $x \in \mathbb{F}_q$ such that 
$x^2+x+\frac{(z\beta^3 + z\beta^2 + z + \beta)(z+\beta)}{z^2\beta^4}=0.$

Consider the plane $\mathbb{F}_q$-rational curve 
$$\mathcal{C}: X^2+X+\frac{(Z\beta^3 + Z\beta^2 + Z + \beta)(Z+\beta)}{Z^2\beta^4}=0.$$
It is birationally equivalent to 
$$\left(X+\frac{1}{\beta Z}\right)^2+X+\frac{1}{\beta Z}+\frac{(Z\beta^3 + Z\beta^2 + Z + \beta)(Z+\beta)}{Z^2\beta^4}=0,$$
that is 
$$X^2+X+\frac{\beta^3+\beta^2+1}{\beta^4}+\frac{1}{Z}=0,$$
which is clearly absolutely irreducible and of genus at most 3. This means that there are at least $q-2\sqrt{q}$ $\mathbb{F}_q$-rational points in $\mathcal{C}$ and thus $\frac{q-2\sqrt{q}}{2}-1$ values $z\in \mathbb{F}_q$ for which \[{\rm Tr}_{q/2}\left(\frac{(z\beta^3 + z\beta^2 + z + \beta)(z+\beta)}{z^2\beta^4}\right)=0.\] Among these values,  at most 10 satisfy $r_0(z)=0$. Since $q\geq64$, $\Delta_0\neq \emptyset$.
\end{proof}

\begin{proof}[Proof of Lemma \ref{Prop:alpha_1 beta_Fq}]
We will prove that for each $\beta$ there exists $z\in \mathbb{F}_{q}$ such that the point $(1: z_0:\beta)$ is not $1$-saturated, by proving that the set 
$$\{(x^q+x^{q^3}+z x,x^{q^2}+\beta x^{q^3}+\beta x): x \in \mathbb{F}_{q^4} \}\subseteq \PG(1,q^4)$$
is scattered. 

Let us consider $z\in \Delta_0\neq \emptyset$. First we prove that the point $(1:0)\in \PG(1,q^4)$ is of weight $0$.  This can be readily seen as the Dickson matrix of $x^{q^2}+\beta x^{q^3}+\beta x$ is non-singular. 

We consider now a point $(m:1)\in \PG(1,q^4)$ with $m\in \mathbb{F}_{q^4}$.  Such a point is of weight at most one if and only if the rank of 
 $$M:=
\begin{pmatrix}
m \beta+z&1&m&1+m \beta\\
	1+mq \beta&m^q \beta+z&1&m^q\\
	m^{q^2}&1+m^{q^2} \beta&m^{q^2} \beta+z&1\\
	1&m^{q^3}&1+m^{q^3} \beta&m^{q^3} \beta+z\\
\end{pmatrix}
$$
is at least three.
Let $f_1=\det(N_1)$ and $f_2=\det(N_2)$ be the determinants of the north-right and south-right $3\times 3$ submatrix of $M$, respectively.

Then 
\begin{eqnarray*}
f_1 &:=& \beta^3 m^{q^3+q^2+1} + (z\beta + \beta^2)m^{q^2+1} + (z\beta + \beta^2 + \beta + 1)m^{q^3+1}\\
&&+ (z + 
    \beta)m + \beta m^{q^3+q^2} + (z\beta + \beta)m^{q^2} + (z\beta +z)m^{q^3} + z^2;\\
f_2 &:=& (\beta^3 + \beta^2 + 1)m^{q^3+q^2+1} + (\beta^2 + \beta)m^{q^2+1} + (z\beta + \beta^2)m^{q^3+1} \\
&&+
    (\beta + 1)m + (z\beta^2 + z\beta + \beta)m^{q^3+q^2} + (z\beta + \beta + 1)m^{q^2} + (z^2 +
    z\beta)m^{q^3}.\\
\end{eqnarray*}

Using $f_1=0$ and $(f_1)^{q^3}=0$ to eliminate $m^{q^3}$ and $m^{q^2}$ from $(f_2)^{q^3}=0$, one gets

$$((\beta^2 + \beta) m^{q+1} + \beta m + (z+\beta)m^q + 1)((z\beta^3 + \beta + 1)m^{q+1} +(z^2\beta + z)m^q  +(z^2\beta + z)m +   z^2\beta + z^2)=0.$$

\begin{itemize}
    \item Suppose that $(\beta^2 + \beta) m^{q+1} + \beta m + (z+\beta)m^q + 1=0$. 
    Then $(m \beta^2 + m \beta +z + \beta)m^q= m\beta +   1$. So $m\neq 0$. Also, $m \beta^2 + m \beta +z + \beta=0$ would imply $m\beta +   1=0$ and thus $z=1$, a contradiction. Since $m\in \mathbb{F}_{q^4}$ this yields $z^3(m^2\beta^2 + m^2\beta + m z + 1)=0$. Using $m^q= \frac{m\beta +  + 1}{m \beta^2 + m \beta +z + \beta}$ again in $(f_2)^{q^3}=0$ one gets
$$(m\beta + z)(m\beta^2 + m\beta + m + z + \beta)(m z\beta^2 + mz\beta + m\beta + z^2 + \beta + 1)=0.$$

Since $z\notin  \mathbb{F}_4 \cup \{\beta\}$, none of these three factors vanishes. Also, combining each of them with  $m^2\beta^2 + m^2\beta + m z + 1=0$, necessarily 
$$\beta z(z+1)=0,$$
or $$z(z^2 + z\beta^3 + z\beta^2 + z\beta + \beta^3 + \beta^2 + 1)=0,$$
or $$\beta z(z+1)(z^2 + z\beta^3 + z\beta^2 + z\beta + \beta^3 + \beta^2 + 1)=0,$$
a contradiction to $z\in \Delta_0$.
\item Suppose that $(z\beta^3 + \beta + 1)m^{q+1} +(z^2\beta + z)m^q  +(z^2\beta + z)m +   z^2\beta + z^2=0$. Since $z\in \Delta_0$, $z\beta^3 + \beta + 1\neq 1$. Then $(m z\beta^3 + m\beta + m + z^2\beta + z)m^q= m z^2\beta + m z  + z^2\beta + z^2$. Clearly $m=0$ is not possible. Also $m z\beta^3 + m\beta + m + z^2\beta + z=0$ yields $m z^2\beta + m z  + z^2\beta + z^2=0$ and thus $z \beta (z^2 + z\beta^2 + z\beta + 1)=0$, a contradiction to $z \in \Delta_0$. From $m^q= \frac{m z^2\beta + m z  + z^2\beta + z^2}{m z\beta^3 + m\beta + m + z^2\beta + z}$, substituting it in $\det(M)=0$ one gets
$$\beta z ((z\beta^3 + z\beta^2 + z + \beta)m^2 + z^2\beta^2 m+ z^3 + z^2\beta)=0.$$
Note that $(z\beta^3 + z\beta^2 + z + \beta)m^2 + z^2\beta^2 m+ z^3 + z^2\beta=0$ is an equation in $m$ defined over $\mathbb{F}_q$. Since ${\rm Tr}_{q/2}\left(\frac{(z\beta^3 + z\beta^2 + z + \beta)(z+\beta)}{z^2\beta^4}\right)=0$, $m\in \mathbb{F}_q$ and thus 
    $$\frac{m z^2\beta + m z  + z^2\beta + z^2}{m z\beta^3 + m\beta + m + z^2\beta + z}=m.$$
Together with $(z\beta^3 + z\beta^2 + z + \beta)m^2 + z^2\beta^2 m+ z^3 + z^2\beta=0$, this gives
    $$\beta z (z^2 + z\beta^2 + z\beta + 1)=0,$$
    a contradiction to $z \in \Delta_0$.
\end{itemize}
\end{proof}

Consider the following polynomials
\begin{eqnarray*}
h_1(z) &:=&\beta^{q+1}z(z+1)(z^2 + (\beta +\beta^{q^3})z + \beta + \beta^{q^3});\\
h_2 (z) &:=& z^5 + (\beta^{q^2+q+1} + \beta^{q+1} + \beta + \beta^{q^2} + 
        \beta^{q^3})z^4 + (\beta^{q^2+q+2} + \beta^{q+2} + 
        \beta^{q^2+2q+1} + \beta^{2q+1}\\&& + \beta^{q^3+q+1}+ 
        \beta^{q^2+1} + \beta + \beta^{q^2}+ \beta^{q^2+q} + \beta^q 
        + \beta^{q^2} + \beta^{q^3} + 1)z^3\\&& + ( \beta^{q^2+q+2} + 
        \beta^{q^3+2q+1}+ \beta^{2q+1}+ \beta^{q+1}+ 
        \beta^{q^3+2q} + \beta^{2q}+ \beta^{q^2+q}+ \beta^{q^3+q}+
        \beta^q + \beta^{q^2})z^2\\&& + ( \beta^{q^2+q+2} + \beta^{q+2}+ 
         \beta^{q^2+2q+1} +  \beta^{q^3+q+1} +  \beta^{q^2+1} + 
        \beta^{q^3+q}+ \beta^q + \beta^{q^2})z\\&& +  \beta^{q^2+q+2} + 
        \beta^{q^3+2q+1}+ \beta^{q^2+q+1}+ \beta^{q^3+2q};\\
h_3(z) &:=&\beta (z+1)
    (z^5 + (\beta +  \beta^{q^3+q^2+q} + \beta^{q^2+q}+ \beta^q + 
        \beta^{q^3})z^4 + ( \beta^{q^2+q+1}+ \beta + 
         \beta^{q^3+q^2+2q}\\&& +  \beta^{q^2+2q} + 
         \beta^{q^3+2q^2+q} +  \beta^{2q^2+q} + \beta^{q^3+q}+ 
        \beta^q + \beta^{2q^2} + \beta^{q^3+q^2} + \beta^{q^2} + \beta^{q^3} +
        1)z^3\\&& + ( \beta^{q^2+q+1} +  \beta^{2q^2+1} +  \beta^{q^2+1} + 
         \beta^{q^3+q^2+2q} +  \beta^{2q^2+q} + \beta^{q^2+q}+ 
        \beta^{2q^2} + \beta^{q^3+q^2} + \beta^{q^2} + \beta^{q^3})z^2\\&& + 
        ( \beta^{q^2+q+1}+  \beta^{q^2+1} +  \beta^{q^3+q^2+2q} + 
         \beta^{q^2+2q} +  \beta^{q^3+2q^2+q} + \beta^{q^3+q}+ \beta^{q^2} +
        \beta^{q^3})z\\&& +  \beta^{q^2+q+1} +  \beta^{2q^2+1} +  \beta^{q^3+q^2+2q} +
         \beta^{q^3+q^2+q}\\
h_4(z) &:=&{\rm Tr}_{q^4/q}(\beta+\beta^{q+1})z^3 + ({\rm Tr}_{q^4/q}(\beta^{q^+q+1})  + \beta^{q^2+1} +\beta^{q^3+q})z^2 + {\rm Tr}_{q^4/q}(\beta)z + {\rm Tr}_{q^4/q}(\beta^{q+1}).
\end{eqnarray*}

Let $$\Delta_1:= \{z \in \mathbb{F}_q : z(z+1)(
    z\beta^{q^2+1} + \beta + \beta^{q^2})(z^2 + z\beta^{q^2+q} + z\beta^{q^2} + \beta^q + \beta^{q^2} + 1)h_1(z)h_2(z)hg_3(z)h_4(z)\neq 0\}.$$

\begin{proof}[Proof of Lemma \ref{lem:second-case}]
The proof is similar to the one of Lemma \ref{Prop:alpha_1 beta_Fq}. First note that since $\beta^{q^3+q}=\beta^{q^2+q}$, $\beta\notin \mathbb{F}_q$ yields $\beta\notin \mathbb{F}_{q^2}$. We will prove that for each such $\beta$ there exists $z\in \Delta_1$ such that the point $(1: z:\beta)$ is not $1$-saturated. This holds if the set 
$$\{(x^q+x^{q^3}+z x,x^{q^2}+\beta x^{q^3}+\beta x): x \in \mathbb{F}_{q^4} \}\subseteq \PG(1,q^4)$$
is scattered. Arguing as in the proof of Lemma \ref{Prop:alpha_1 beta_Fq}, it can be shown that the point $(0:1)$ is of weight $1$. To investigate the weight of a point $(1:m)$, $m\in \mathbb{F}_{q^4}$, we consider the rank of 
 $$M:=
\begin{pmatrix}
m \beta+z&1&m&1+m \beta\\
	1+m^q \beta^q&m^q \beta^q+z^q&1&m^q\\
	m^{q^2}&1+m^{q^2} \beta^{q^2}&m^{q^2} \beta^{q^2}+z^{q^2}&1\\
	1&m^{q^3}&1+m^{q^3} \beta^{q^3}&m^{q^3} \beta^{q^3}+z^{q^3}\\
\end{pmatrix}.
$$
We will prove that for each choice of $\beta$ the rank of $M$ is at least three independently of $m\in \mathbb{F}_{q^4}$. 
Let $f_1=\det(N_1)$ and $f_2=\det(N_2)$ be the determinants of the north-right and south-right $3\times 3$ submatrix of $M$, respectively. In this case, 
\begin{eqnarray*}
f_1(z) &:=&(\beta^{q^3+q^2+q}+ \beta^{q^2+q} + \beta^{q^3+q})m^{q^3+q^2+q} +  (z\beta^q + 
    \beta^{q^3+q})m^{q^3+q}+ (z\beta^{q^2}+ z)m^{q^2}\\
    && + \beta^q m^{q^2+q}+(z\beta^{q^3} + 
    \beta^{q^3+q^2} + \beta^{q^2} + 1)m^{q^3+q^2} +  (z\beta^q + \beta^q)m^q + 
    (z + \beta^{q^3})m^{q^3} + z^2;\\
f_2(z) &:=&
(\beta^{q^3+q^2+q} + \beta^{q^3+q} + 1)m^{q^3+q^2+q} + 
    (z\beta^{q^2+q} + z\beta^q + \beta^{q^2})m^{q^2+q} + (\beta^{q^3+q}
    + \beta^q)m^{q^3+q}\\&& + (z\beta^q + \beta^q + 1)m^q + (z\beta^{q^3} + 
    \beta^{q^3+q^2})m^{q^3+q^2} + (z^2+z\beta^{q^2})m^{q^2} +  (\beta^{q^3} + 1)m^{q^3}.
\end{eqnarray*}
Using $f_1=0$ and $(f_1)^{q^3}=0$ to eliminate $m^{q^3}$ and $m^{q^2}$ from $(f_2)^{q^3}=0$, one gets

$$(\beta^{q+1} + \beta)m^{q+1} + \beta m + (z + \beta^q)m^q + 1=0$$
or
\begin{eqnarray*}
&&(z\beta^{q^2+q+1} + z\beta^{q+1} + z\beta^{q^2+1} + 
\beta^{q^2+1} + \beta^{q^2+q}+ \beta^q + 1)m^{q+1}\\
&&+ (z^2\beta + z + \beta + \beta^{q^2})m + (z^2\beta^{q^2} + z\beta^q + z\beta^{q^2}+ 
        z)m^q + z^2\beta^{q^2}+ z^2=0.
        \end{eqnarray*}

\begin{itemize}
    \item Suppose that   $(\beta^{q+1} + \beta)m^{q+1} + \beta m + (z + \beta^q)m^q + 1=0$. Then $(m \beta^{q+1} + m \beta +z + \beta^q)m^q= m\beta +   1$. Clearly $m\neq 0$. Also, $m \beta^{q+1} + m \beta +z + \beta^q=0$ would imply $m\beta +   1=0$ and thus $z=1$, a contradiction. Since $m\in \mathbb{F}_{q^4}$, this yields 
    \begin{eqnarray}
        &&\Big((\beta^{q+1} + \beta)z^3 +(\beta^{q+2} + \beta^2 + 
    \beta^{q^3+q+1} + \beta^{q+1} + \beta^{q^2+1} + 
    \beta^{q^3+1})z^2\nonumber\\
    && + (\beta^{q+2} + \beta^{q^2+2} + 
    \beta^{q^2+1} + \beta^{q^3+1})z+ \beta^{q^2+2} + 
    \beta^{q^3+q+1}\Big)m^2\nonumber\\
    && + \Big(z^4 + {\rm Tr}_{q^4/q}(\beta)z^3 + ({\rm Tr}_{q^4/q}(\beta)+\beta^{q^2+1}+\beta^{q^3+q}) z^2 + \beta^{q^2+1}+\beta^{q^3+q}\Big)m\nonumber\\
    && +
    z^3 + z^2\beta^q + z^2\beta^{q^2} + z\beta^q + z\beta^{q^2}=0.\label{Eq:SecondoGrado}
    \end{eqnarray}
    Using $m^q= \frac{m\beta   + 1}{m \beta^{q+1} + m \beta +z + \beta^q}$ again in $(f_2)^{q^3}=0$ one gets
$$(m\beta + z)(m\beta^{q+1} + m\beta + m + z + \beta^q)(m z\beta^{q+1} + mz\beta + m\beta^{q+1} + m\beta^{q^2+1} + m\beta + z^2 + 
        z\beta^q + z\beta^{q^2} + \beta^q + 1)=0.$$

Since $\beta \notin \mathbb{F}_q $, none of these three factors vanishes. Also, combining each of them with  \eqref{Eq:SecondoGrado}, necessarily 
$$h_1(z)h_2(z)h_3(z)=0,$$
a contradiction to $z\in \Delta_1$.

\item Suppose that 
\begin{eqnarray*}
&&(z\beta^{q^2+q+1} + z\beta^{q+1} + z\beta^{q^2+1} + 
\beta^{q^2+1} + \beta^{q^2+q}+ \beta^q + 1)m^{q+1}\\
&&+ (z^2\beta + z + \beta + \beta^{q^2})m + (z^2\beta^{q^2} + z\beta^q + z\beta^{q^2}+ 
        z)m^q + z^2\beta^{q^2}+ z^2=0.
        \end{eqnarray*} 
        Then $((z\beta^{q^2+q+1} + z\beta^{q+1} + z\beta^{q^2+1} + 
\beta^{q^2+1} + \beta^{q^2+q}+ \beta^q + 1)m+(z^2\beta^{q^2} + z\beta^q + z\beta^{q^2}+ 
        z))m^q=  (z^2\beta + z + \beta + \beta^{q^2})m+z^2\beta^{q^2}+ z^2$. Clearly $m=0$ is not possible. 
        
        Also $(z\beta^{q^2+q+1} + z\beta^{q+1} + z\beta^{q^2+1} + 
\beta^{q^2+1} + \beta^{q^2+q}+ \beta^q + 1)m+(z^2\beta^{q^2} + z\beta^q + z\beta^{q^2}+ 
        z)=0$ yields $(z^2\beta + z + \beta + \beta^{q^2})m+z^2\beta^{q^2}+ z^2=0$ and thus $z(
    z\beta^{q^2+1} + \beta + \beta^{q^2})(z^2 + z\beta^{q^2+q} + z\beta^{q^2} + \beta^q + \beta^{q^2} + 1)=0$, a contradiction to $z\in \Delta_1$.

        Thus 
    $$m^q =\frac{(z^2\beta + z + \beta + \beta^{q^2})m+z^2\beta^{q^2}+ z^2}{(z\beta^{q^2+q+1} + z\beta^{q+1} + z\beta^{q^2+1} + 
\beta^{q^2+1} + \beta^{q^2+q}+ \beta^q + 1)m+(z^2\beta^{q^2} + z\beta^q + z\beta^{q^2}+ 
        z)}$$
        and since $m\in \mathbb{F}_{q^4}$, necessarily $ h_4(z)(a_2(z)m^2+a_1(z)m+a_0(z))=0$, where
        \begin{eqnarray*}
        a_2(z) &:=& (\beta^{q^3+q^2+q+2}+ \beta^{q^3+q^3+2} + \beta^{q+1})z^3+(\beta^{q^2+q+2} + \beta^{q^3+q^2+2} + \beta^{q^2+2} + \beta^2\\
        &&+ \beta^{q^2+q+1} + 
        \beta^{q+1} +\beta^{q^3+q^2+1}+ \beta^{q^2+1} + \beta^{q^3+1} + \beta + \beta^q)z^2\\
        &&+(\beta^{q^3+q+2} + \beta^{q^2+2} + \beta^{q^3+q^2+q+1} + \beta^{q^2+q+1} + \beta^{q+1}+ \beta^{q^3+q^2+1}+ \beta^{q^2+1}\\
        &&+ \beta^{q^3+q^2+q}+ \beta^q + \beta^{q^3})z+\beta^{q+2} + \beta^2 + \beta^{q^2+q+1} + \beta^{q^3+q+1}+ \beta^{q^2+1} + \\
        &&\beta + \beta^{q^3+q^2+q} + \beta^{q^2};\\
         a_1(z) &:=&\beta^{q^3+q^2+q+1}z^4 + {\rm Tr}_{q^4/q}(\beta+\beta^{q+1}+\beta^{q^2+q+1}) z^3\\
         &&+       ({\rm Tr}_{q^4/q}(\beta^{q+1}+\beta^{q^2+q+1})+\beta^{q^2+1}+\beta^{q^3+q})z^2 + {\rm Tr}_{q^4/q}(\beta)z + {\rm Tr}_{q^4/q}(\beta^{q+1});\\
    a_0(z) &:=& z^2
    \Big(\beta^{q^3+q^2}z^3 + (\beta^{q^3+q^2+q} + \beta^{q^2} + \beta^{q^3})z^2 + 
        (\beta^{q^3+q}+ \beta^q + \beta^{q^3+q^2} + \beta^{q^3})z\\
        &&+ \beta^{q^2+q} + 
        \beta^q + \beta^{q^3+q^2}+ \beta^{q^3}\Big)
        \end{eqnarray*}
Since $\beta\notin \mathbb{F}_{q}$ (and thus not belonging to $\mathbb{F}_{q^2}$) the polynomial $a_2(z)m^2+a_1(z)m+a_0(z)$ is nonvanishing, since $a_2(z)\not\equiv 0$. Also,the resultant between $a_2(z)m^2+a_1(z)m+a_0(z)$ and $f_2$ with respect to $m$ is a polynomial of degree $38$ in $z$ whose leading coefficient is   $\beta^{q^3+3q^2+2q^2+4}(\beta^{q}+\beta^{q^2}+\beta^{q^2+q}+\beta^{q^3+q^2})$, which is never zero, since $\beta\notin \mathbb{F}_q$.
Since $q\geq 64$ it is always possible to choose $z\in \Delta_1$ to be not a root of such a polynomial. 
\end{itemize}
The claim follows.
\end{proof}

\end{document}